\documentclass[12pt]{amsart}
\usepackage{amsmath}
\usepackage{amsfonts}
\usepackage{amssymb}
\usepackage{bbding}
\usepackage{stmaryrd}
\usepackage{txfonts}
\usepackage{graphicx}
\usepackage{epsfig}
\usepackage{xypic}
\usepackage{tikz}
\usepackage{longtable}
\usepackage[all]{xy}
\usepackage{pgflibraryarrows}
\usepackage{pgflibrarysnakes}
\usepackage[shortlabels]{enumitem}
\usepackage{ifpdf}
\ifpdf
  \usepackage[colorlinks,final,backref=page,hyperindex]{hyperref}
\else
  \usepackage[colorlinks,final,backref=page,hyperindex,hypertex]{hyperref}
\fi
\usepackage{xspace}

\begin{document}  

\newcommand{\nc}{\newcommand}
\nc{\on}{\operatorname}
\newcommand{\delete}[1]{}

\nc{\mlabel}[1]{\label{#1}}  
\nc{\mcite}[1]{\cite{#1}}  
\nc{\mref}[1]{\ref{#1}}  
\nc{\mbibitem}[1]{\bibitem{#1}} 

\delete{
\nc{\mlabel}[1]{\label{#1}  
{\hfill \hspace{1cm}{\bf{{\ }\hfill(#1)}}}}
\nc{\mcite}[1]{\cite{#1}{{\bf{{\ }(#1)}}}}  
\nc{\mref}[1]{\ref{#1}{{\bf{{\ }(#1)}}}}  
\nc{\mbibitem}[1]{\bibitem[\bf #1]{#1}} 
}

\newtheorem{theorem}{Theorem}[section]
\newtheorem{prop}[theorem]{Proposition}

\newtheorem{lemma}[theorem]{Lemma}
\newtheorem{coro}[theorem]{Corollary}
\newtheorem{prop-def}{Proposition-Definition}[section]
\newtheorem{claim}{Claim}[section]
\newtheorem{propprop}{Proposed Proposition}[section]
\newtheorem{conjecture}{Conjecture}
\newtheorem{assumption}{Assumption}
\newtheorem{condition}[theorem]{Assumption}
\newtheorem{question}[theorem]{Question}
\theoremstyle{definition}
\newtheorem{defn}[theorem]{Definition}
\newtheorem{remark}[theorem]{Remark}
\newtheorem{exam}[theorem]{Example}
\renewcommand{\labelenumi}{{\rm(\alph{enumi})}}
\renewcommand{\theenumi}{\alph{enumi}}


\nc{\homp}{P}
\nc{\pprime}{\frakQ}

\nc{\adec}{\check{;}}
\nc{\dftimes}{\widetilde{\otimes}} \nc{\dfl}{\succ}
\nc{\dfr}{\prec} \nc{\dfc}{\circ} \nc{\dfb}{\bullet}
\nc{\dft}{\star} \nc{\dfcf}{{\mathbf k}} \nc{\spr}{\cdot}
\nc{\disp}[1]{\displaystyle{#1}}
\nc{\bin}[2]{ (_{\stackrel{\scs{#1}}{\scs{#2}}})}  
\nc{\binc}[2]{ \left (\!\! \begin{array}{c} \scs{#1}\\
    \scs{#2} \end{array}\!\! \right )}  
\nc{\bincc}[2]{  \left ( {\scs{#1} \atop
    \vspace{-.5cm}\scs{#2}} \right )}  
\nc{\sarray}[2]{\begin{array}{c}#1 \vspace{.1cm}\\ \hline
    \vspace{-.35cm} \\ #2 \end{array}}
\nc{\bs}{\bar{S}} \nc{\dcup}{\stackrel{\bullet}{\cup}}
\nc{\dbigcup}{\stackrel{\bullet}{\bigcup}} \nc{\etree}{\big |}
\nc{\la}{\longrightarrow} \nc{\fe}{\'{e}} \nc{\rar}{\rightarrow}
\nc{\dar}{\downarrow} \nc{\dap}[1]{\downarrow
\rlap{$\scriptstyle{#1}$}} \nc{\uap}[1]{\uparrow
\rlap{$\scriptstyle{#1}$}} \nc{\defeq}{\stackrel{\rm def}{=}}
\nc{\diffa}[1]{\{#1\}}
\nc{\diffs}[1]{\Delta{#1}}
\nc{\rbring}[1]{_{RB}( #1 )}
\nc{\rrbo}{R_{RB}\langle Q\rangle}
\nc{\dis}[1]{\displaystyle{#1}} \nc{\dotcup}{\,
\displaystyle{\bigcup^\bullet}\ } \nc{\sdotcup}{\tiny{
\displaystyle{\bigcup^\bullet}\ }} \nc{\hcm}{\ \hat{,}\ }
\nc{\hcirc}{\hat{\circ}} \nc{\hts}{\hat{\shpr}}
\nc{\lts}{\stackrel{\leftarrow}{\shpr}}
\nc{\rts}{\stackrel{\rightarrow}{\shpr}} \nc{\lleft}{[}
\nc{\lright}{]} \nc{\uni}[1]{\tilde{#1}} \nc{\wor}[1]{\check{#1}}
\nc{\free}[1]{\bar{#1}} \nc{\den}[1]{\check{#1}} \nc{\lrpa}{\wr}
\nc{\curlyl}{\left \{ \begin{array}{c} {} \\ {} \end{array}
    \right .  \!\!\!\!\!\!\!}
\nc{\curlyr}{ \!\!\!\!\!\!\!
    \left . \begin{array}{c} {} \\ {} \end{array}
    \right \} }
\nc{\leaf}{\ell}       
\nc{\longmid}{\left | \begin{array}{c} {} \\ {} \end{array}
    \right . \!\!\!\!\!\!\!}
\nc{\ot}{\otimes}
\nc{\oot}{\boxtimes}
\nc{\sot}{{\scriptstyle{\ot}}}
\nc{\otm}{\overline{\ot}}
\nc{\ora}[1]{\stackrel{#1}{\rar}}
\nc{\ola}[1]{\stackrel{#1}{\la}}
\nc{\scs}[1]{\scriptstyle{#1}} \nc{\mrm}[1]{{\rm #1}}
\nc{\margin}[1]{\marginpar{\rm #1}}   
\nc{\dirlim}{\displaystyle{\lim_{\longrightarrow}}\,}
\nc{\invlim}{\displaystyle{\lim_{\longleftarrow}}\,}
\nc{\mvp}{\vspace{0.5cm}} \nc{\svp}{\vspace{2cm}}
\nc{\vp}{\vspace{8cm}} \nc{\proofbegin}{\noindent{\bf Proof: }}
\nc{\proofend}{$\blacksquare$ \vspace{0.5cm}}
\nc{\oprod}{\odot}
\nc{\sha}{{\mbox{\cyr X}}}  
\nc{\ncsha}{{\mbox{\cyr X}^{\mathrm NC}}} \nc{\ncshao}{{\mbox{\cyr
X}^{\mathrm NC,\,0}}}
\nc{\shpr}{\diamond}    
\nc{\shprm}{\overline{\diamond}}    
\nc{\shpro}{\diamond^0}    
\nc{\shprr}{\diamond^r}     
\nc{\shpra}{\overline{\diamond}^r}
\nc{\shpru}{\check{\diamond}} \nc{\catpr}{\diamond_l}
\nc{\rcatpr}{\diamond_r} \nc{\lapr}{\diamond_a}
\nc{\sqcupm}{\ot}
\nc{\lepr}{\diamond_e} \nc{\vep}{\varepsilon} \nc{\labs}{\mid\!}
\nc{\rabs}{\!\mid} \nc{\hsha}{\widehat{\sha}}
\nc{\lsha}{\stackrel{\leftarrow}{\sha}}
\nc{\rsha}{\stackrel{\rightarrow}{\sha}} \nc{\lc}{\lfloor}
\nc{\rc}{\rfloor} \nc{\sqmon}[1]{\langle #1\rangle}
\nc{\forest}{\calf} \nc{\ass}[1]{\alpha({#1})}
\nc{\altx}{\Lambda_X} \nc{\vecT}{\vec{T}} \nc{\onetree}{\bullet}
\nc{\Ao}{\check{A}}
\nc{\seta}{\underline{\Ao}}
\nc{\deltaa}{\overline{\delta}}
\nc{\trho}{\tilde{\rho}}
\nc{\tpow}[2]{{#2}^{\ot #1}}

\nc{\mmbox}[1]{\mbox{\ #1\ }} \nc{\ann}{\mrm{ann}}
\nc{\Aut}{\mrm{Aut}}
\nc{\bread}{\mrm{b}}
\nc{\can}{\mrm{can}} \nc{\colim}{\mrm{colim}}
\nc{\Cont}{\mrm{Cont}} \nc{\rchar}{\mrm{char}}
\nc{\cok}{\mrm{coker}} \nc{\dtf}{{R-{\rm tf}}} \nc{\dtor}{{R-{\rm
tor}}}
\renewcommand{\det}{\mrm{det}}
\nc{\depth}{{\mrm d}}
\nc{\Div}{{\mrm Div}} \nc{\End}{\mrm{End}} \nc{\Ext}{\mrm{Ext}}
\nc{\Fil}{\mrm{Fil}} \nc{\Frob}{\mrm{Frob}} \nc{\Gal}{\mrm{Gal}}
\nc{\GL}{\mrm{GL}} \nc{\Hom}{\mrm{Hom}} \nc{\hsr}{\mrm{H}}
\nc{\hpol}{\mrm{HP}} \nc{\id}{\mrm{id}} \nc{\im}{\mrm{im}}
\nc{\incl}{\mrm{incl}} \nc{\length}{\mrm{length}}
\nc{\LR}{\mrm{LR}} \nc{\mchar}{\rm char} \nc{\NC}{\mrm{NC}}
\nc{\Mod}{\operatorname{-Mod}}
\nc{\mpart}{\mrm{part}} \nc{\pl}{\mrm{PL}}
\nc{\ql}{{\QQ_\ell}} \nc{\qp}{{\QQ_p}}
\nc{\rank}{\mrm{rank}} \nc{\rba}{\rm{RBA }} \nc{\rbas}{\rm{RBAs }}
\nc{\rbpl}{\mrm{RBPL}}
\nc{\rbw}{\rm{RBW }} \nc{\rbws}{\rm{RBWs }} \nc{\rcot}{\mrm{cot}}
\nc{\rest}{\rm{controlled}\xspace}
\nc{\rdef}{\mrm{def}} \nc{\rdiv}{{\rm div}} \nc{\rtf}{{\rm tf}}
\nc{\rtor}{{\rm tor}} \nc{\res}{\mrm{res}} \nc{\SL}{\mrm{SL}}
\nc{\Spec}{\mrm{Spec}} \nc{\tor}{\mrm{tor}} \nc{\Tr}{\mrm{Tr}}
\nc{\mtr}{\mrm{sk}}

\nc{\ab}{\mathbf{Ab}} \nc{\Alg}{\mathbf{Alg}}
\nc{\Algo}{\mathbf{Alg}^0} \nc{\Bax}{\mathbf{Bax}}
\nc{\Baxo}{\mathbf{Bax}^0} \nc{\Dif}{\mathbf{Dif}}
\nc{\CDif}{\mathbf{CDif}}
\nc{\CRB}{\mathbf{CRB}}
\nc{\CDRB}{\mathbf{CDRB}}
\nc{\RB}{\mathrm{RB}}
\nc{\RSD}{\operatorname{RSD}}
\nc{\DRB}{\mathbf{DRB}}
\nc{\RBo}{\mathbf{RB}^0} \nc{\BRB}{\mathbf{RB}}
\nc{\Dend}{\mathbf{DD}}
\nc{\Set}{\mathbf{Set}}
\nc{\bfk}{{\bf k}} \nc{\bfone}{{\bf 1}}
\nc{\base}[1]{{a_{#1}}} \nc{\detail}{\marginpar{\bf More detail}
    \noindent{\bf Need more detail!}
    \svp}
\nc{\Diff}{\mathbf{Diff}} \nc{\gap}{\marginpar{\bf
Incomplete}\noindent{\bf Incomplete!!}
    \svp}
\nc{\FMod}{\mathbf{FMod}} \nc{\mset}{\mathbf{MSet}}
\nc{\rb}{\mathrm{RB}} \nc{\Int}{\mathbf{Int}}
\nc{\Mon}{\mathbf{Mon}}
\nc{\oR}{\overline{R}}
\nc{\remarks}{\noindent{\bf Remarks: }} \nc{\Rep}{\mathbf{Rep}}
\nc{\Rings}{\mathbf{Rings}} \nc{\Sets}{\mathbf{Sets}}
\nc{\DT}{\mathbf{DT}}

\nc{\bbA}{{\mathbb A}} \nc{\CC}{{\mathbb C}} \nc{\DD}{{\mathbb D}}
\nc{\EE}{{\mathbb E}} \nc{\FF}{{\mathbb F}} \nc{\GG}{{\mathbb G}}
\nc{\HH}{{\mathbb H}} \nc{\LL}{{\mathbb L}} \nc{\NN}{{\mathbb N}}
\nc{\QQ}{{\mathbb Q}} \nc{\RR}{{\mathbb R}} \nc{\TT}{{\mathbb T}}
\nc{\VV}{{\mathbb V}} \nc{\ZZ}{{\mathbb Z}}


\nc{\calA}{{\mathcal A}} \nc{\calC}{{\mathcal C}}
\nc{\calD}{{\mathcal D}} \nc{\calE}{{\mathcal E}}
\nc{\calF}{{\mathcal F}} \nc{\calfr}{{{\mathcal F}^{\,r}}}
\nc{\calfo}{{\mathcal F}^0} \nc{\calfro}{{\mathcal F}^{\,r,0}}
\nc{\oF}{\overline{F}}  \nc{\calG}{{\mathcal G}}
\nc{\calH}{{\mathcal H}} \nc{\calI}{{\mathcal I}}
\nc{\calJ}{{\mathcal J}} \nc{\calL}{{\mathcal L}}
\nc{\calM}{{\mathcal M}} \nc{\calN}{{\mathcal N}}
\nc{\calO}{{\mathcal O}} \nc{\calP}{{\mathcal P}} \nc{\calp}{{\mathcal P}}
\nc{\calR}{{\mathcal R}} \nc{\calS}{{\mathcal S}}
\nc{\calT}{{\mathcal T}}\nc{\caltr}{{\mathcal T}^{\,r}}
\nc{\calU}{{\mathcal U}} \nc{\calV}{{\mathcal V}}
\nc{\calW}{{\mathcal W}} \nc{\calX}{{\mathcal X}}
\nc{\calY}{{\mathcal Y}}\nc{\calZ}{{\mathcal Z}}

\nc{\fraka}{{\mathfrak a}} \nc{\frakB}{{\mathfrak B}}
\nc{\frakb}{{\mathfrak b}} \nc{\frakc}{{\mathfrak c}}
\nc{\frakd}{{\mathfrak d}}
\nc{\oD}{\overline{D}}
\nc{\frakF}{{\mathfrak F}} \nc{\frakg}{{\mathfrak g}}
\nc{\frakm}{{\mathfrak m}} \nc{\frakM}{{\mathfrak M}}
\nc{\frakMo}{{\mathfrak M}^0} \nc{\frakp}{{\mathfrak p}}
\nc{\frakQ}{{\mathcal Q}}
\nc{\frakS}{{\mathfrak S}} \nc{\frakSo}{{\mathfrak S}^0}
\nc{\fraks}{{\mathfrak s}} \nc{\os}{\overline{\fraks}}
\nc{\frakT}{{\mathfrak T}}
\nc{\oT}{\overline{T}}
\nc{\frakX}{{\mathfrak X}} \nc{\frakXo}{{\mathfrak X}^0}
\nc{\frakx}{{\mathbf x}}
\nc{\frakTx}{\frakT}      
\nc{\frakTa}{\frakT^a}        
\nc{\frakTxo}{\frakTx^0}   
\nc{\caltao}{\calt^{a,0}}   
\nc{\oV}{\overline{V}}
\nc{\ox}{\overline{\frakx}} \nc{\fraky}{{\mathfrak y}}
\nc{\frakz}{{\mathfrak z}} \nc{\oX}{\overline{X}}
\nc{\oZ}{\overline{Z}}

\font\cyr=wncyr10


\title{Representations and modules of Rota-Baxter algebras}
\author{Li Guo}
\address{Department of Mathematics and Computer Science,
Rutgers University, Newark, New Jersey 07102}
\email{liguo@rutgers.edu}
\author{Zongzhu Lin}
\address{Department of Mathematics,
Kansas State University,
Mahattan, Kansas 66506}
\email{zlin@math.ksu.edu}


\begin{abstract}
We give a broad study of representation and module theory of Rota-Baxter algebras. Regular-singular decompositions of Rota-Baxter algebras and Rota-Baxter modules are obtained under the condition of quasi-idempotency. Representations of an Rota-Baxter algebra are shown to be equivalent to the representations of the ring of Rota-Baxter operators whose categorical properties are obtained and explicit constructions are provided. Representations from coalgebras are investigated and their algebraic Birkhoff factorization is given. Representations of Rota-Baxter algebras in the tensor category contexts are also formulated.
\end{abstract}

\subjclass[2010]{16W99, 16D90, 16G99, 16S32,16B10}

\keywords{Rota-Baxter algebra, Rota-Baxter module, ring of Rota-Baxter operators, matrix representation, coalgebra
}

\maketitle

\tableofcontents

\setcounter{section}{0}

\section{Introduction}
The study of Rota-Baxter algebras originated from probability and combinatorics~\mcite{Ba,Ca,Rot1} and has recently found remarkable applications in diverse areas of mathematics and physics, especially in quantum field theory (QFT) through the algebraic approach of Connes and Kreimer to renormalization in perturbative QFT~\mcite{CK1,CM}.

As in the case of common algebraic structures such as associative algebras and Lie algebras, it is important to study the modules and representations of Rota-Baxter algebras. Our interest in pursuing this subject here is foremost motivated by investigations from QFT~\mcite{EGGV,EGs}. There, in the framework of Connes and Kreimer~\mcite{CK1}, one starts with a Hopf algebra (such as the Connes-Kreimer Hopf algebra of Feynman diagrams) and a commutative Rota-Baxter algebra $(A,Q)$ of weight $-1$ (such as the algebra of Laurent series with the projection to the pole part). Then with the convolution product, the space $\Hom(H,A)$ of linear maps is an algebra and, by post-composition, the Rota-Baxter operator $Q$ induces a Rota-Baxter operator $P$ on $\Hom(H,A)$. The Rota-Baxter algebra $(\Hom(H,A),P)$ and its (Atkinson) decomposition encode information of renormalization of QFT. Thus it would be desirable to obtain a more concrete representation of this algebra so that information could be extracted more easily. This is the approach taken in~\mcite{EGGV,EGs} where a representation for the Rota-Baxter algebra $(\Hom(H,A),P)$ is expressed as a matrix Rota-Baxter algebra $M^u_\infty(A)$ where the Rota-Baxter operator is defined entrywise. Another motivations of representations of Rota-Baxter algebra arises from algebraic and differential  geometry, which will be discussed in Section~\mref{sec:rbm}.

For an associative algebra or a Lie algebra, any representation over a vector space can be expressed in the form of a matrix algebra. As we will see in this paper, this is far from being the case for a Rota-Baxter (associative) algebra. Thus our goal of this paper is two fold. On the one hand we start a general study of representations of Rota-Baxter algebras, through modules over a Rota-Baxter algebra and the related ring of Rota-Baxter operators, inspired by the related study of differential algebras and rings of differential operators. One the other hand, we try to understand further the algebraic framework that leads to the matrix representation of Rota-Baxter algebras that arise from the aforementioned applications. Based on this work,\footnote{The present paper was started a few years ago when the two authors, working in Rota-Baxter algebra and representation theory respectively, tried to bring the two subjects together. The paper had its various versions with limited circulations, but was not completed as new and interesting connections showing up. In the mean time, several papers motivated by this paper have appeared~\cite{QGG,LQ,QP,ZGZ}. So to imitate Zariski and Samuel in the introduction of their well-known book~\cite{ZS}, this paper has become the unborn mother of several children. These and other recent developments motivated the present authors to complete the paper while leaving some loose ends to future treatments.} representations of the Rota-Baxter algebra of Laurent series were discussed in~\cite{LQ} where an interesting connection was found with class numbers in algebraic number theory, a similar approach to the Rota-Baxter algebra of polynomial algebra was taken in~\cite{QP}, and derived functors of Rota-Baxter modules were studied in~\cite{QGG}.

Here is an outline of the paper. In Section~\mref{sec:rbm} the concept of a Rota-Baxter module over a Rota-Baxter algebra is introduced and the regular-singular decomposition of a quasi-idempotent Rota-Baxter module is provided. The classical (additive) Atkinson factorization of a Rota-Baxter algebra is generalized to Rota-Baxter modules. The representation of a product Rota-Baxter algebra is discussed in terms of quiver representations. For a given Rota-Baxter algebra, the ring of Rota-Baxter operators on this Rota-Baxter algebra is introduced in Section~\mref{sec:rbo} and its relation with Rota-Baxter modules is established, by an equivalence between the category of Rota-Baxter modules and the category of modules over the ring of Rota-Baxter operators.
In Section~\mref{sec:rbost} we give a construction of the ring of Rota-Baxter operators with more detailed description for the special cases of divided powers and Laurent series.
In Section~\mref{sec:mat} we revisit the topic of matrix representations that motivated our study and give a class of representations of convolution Rota-Baxter algebras by endomorphism and matrix Rota-Baxter algebras. We also prove an algebraic Birkhoff factorization for Rota-Baxter modules. Section~\mref{sec:out} gives a brief discussion on Rota-Baxter algebras in the tensor category context.
\smallskip

\noindent
{\bf Notations.} Throughout this paper, $\bfk$ denotes a unitary commutative ring. All algebras, linear maps and tensor products are taken over $\bfk$ unless otherwise specified. By an algebra we mean a unitary associative algebra while by a nonunitary algebra we mean an associative algebra which might not have an identity.

\section{Rota-Baxter modules and their regular singular decompositions}
\mlabel{sec:rbm}

We first introduce the concept of a Rota-Baxter module with motivation from a differential module. We then give some general properties of Rota-Baxter modules before focusing on the regular-singular decomposition of Rota-Baxter modules over a class of Rota-Baxter algebras.

\subsection{Rota-Baxter modules}
\mlabel{ss:rbm}

\subsubsection{Differential modules} \label{diff:mod}

To further motivate the study of modules over a Rota-Baxter algebra, we recall the well established case of differential algebras for which we refer the reader to~\mcite{Bj} for details. Let $(R,d)$ be a {\bf differential algebra}~\mcite{Kol,SP}, defined to be a pair $(R,d)$ with $R$ a $\bfk$-algebra and $d$ a linear operator on $R$ such that
$$ d(xy)=d(x)y+xd(y) \quad \text{for all } x,y\in R.$$
An $R$-module $M$ is called a {\bf differential module} over $(R, d)$ if there is a linear map $\delta:M\to M$ such that
$$ \delta(a x)=d(a) x + a \delta(x) \quad \text{for all } a\in R, x\in M.$$
In differential geometry, such a $\delta$ is called a connection. See~\cite{HTT} for D-modules in representation theory and algebraic geometry, and \mcite{GK3} for the more general notion of differential algebras with weights.

Algebraically, let $\bfk[d]$ be the polynomial algebra in variable $d$ with the standard Hopf algebra structure coming from the algebra of regular functions on the additive algebraic $k$-group $\mathbb{G}_a$, i.e., with the comultiplication defined by $d\mapsto d\otimes 1+1\otimes d$.
$R$ being a differential algebra is equivalent to $R$ being a $\bfk[d]$-module algebra in the sense that   $R$ is a
$\bfk[d]$-module such that the multiplication map $ R\otimes R\rightarrow R$ is a
$\bfk[d]$-module homomorphism and the map of scalar  $u:\bfk \rightarrow R$ (with $u(\alpha)=\alpha 1$)  is a homomorphism of $\bfk[d]$-modules. Then we can form the smash product
algebra $ R[d]:=R\# \bfk[d]$ with the product $(1\# d)(a\#1)=d(a)\#1+a\# d$~\cite{CMo,Sw}.

As can be easily verified, an $R$-module $M$ is a differential module over the differential algebra $(R,d)$ if and only if $M$ is a module over the smash product algebra $ R\#\bfk[d]$. In particular, the category of all differential modules over a differential algebra $ (R, d) $ is an abelian category with enough projective objects.

As a motivating example, let $X$ be an affine algebraic $\bfk$-variety (with $\bfk=\mathbb{C})$ and $R=\bfk[X]$ be the algebra of regular functions. A vector field $D$ on $X$ is a derivation $D\in \operatorname{Der}_{\bfk}(R)$. An $R$-module $M$ is a quasi-coherent sheaf on $X$. The operator $d: M\rightarrow M$ making $M$ a differential module over $(R, D)$ is a connection of the sheaf along the vector field $D$. In this case, the algebra $R\#\bfk[D]$ is exactly the algebra of differential operators on $X$ generated by $D$. The study of representations of Rota-Baxter algebras to be defined below also has this geometric connection as motivation.

\subsubsection{Rota-Baxter modules}
For a given $\lambda\in \bfk$, a {\bf Rota-Baxter algebra of weight $\lambda$} is defined to be a pair $(R,P)$ with $R$ a $\bfk$-algebra and $P$ a linear operator on $R$ satisfying the {\bf Rota-Baxter axiom}:
\begin{equation}
P(x)P(y)= P(xP(y))+P(P(x)y)+\lambda P(xy) \quad \text{for all } x,y\in R.
\mlabel{eq:rba}
\end{equation}
We will often simply denote $(R,P)$ by $R$ if the operator $P$ is understood from the context. See~\mcite{Ba,Gub,KRY,Rot3} for general discussions of Rota-Baxter algebras.

A homomorphism $\sigma: (R, P)\rightarrow (R', P')$ of Rota-Baxter algebras of the same weight $\lambda$  is a homomorphism $\sigma: R\rightarrow R' $ of $ \bfk$-algebras such that
$P'\circ \sigma=\sigma\circ P$.
We note that if $(R, P)$ is a Rota-Baxter algebra of weight $\lambda$, then $(R, \alpha P)$ is  a Rota-Baxter algebra  of weight $ \alpha \lambda$.

\begin{defn}
Let $(R,P)$ be a Rota-Baxter algebra of weight $ \lambda \in \bfk$. \begin{enumerate}
\item
A (left) {\bf Rota-Baxter module over $(R,P)$} or simply a (left) {\bf $(R,P)$-module} is a pair $(M, p)$ with an $R$-module $M$ and  a linear map $p:M\to M$ satisfying
\begin{equation}
P(a)p(x)=p(ap(x))+p(P(a)x)+\lambda p(ax) \quad \text{for all } a\in R,\ x\in M.
\mlabel{eq:rbm}
\end{equation}
We will simply write $M$ for the pair $(M,p)$ when $p$ is understood.
\item
Let $(M,p_M)$ and $ (N, p_N)$ be two $(R, P)$-modules. A homomorphism $f: (M, p_M)\rightarrow (N, p_N)$ of Rota-Baxter modules is a homomorphism $ f: M\rightarrow N$ of $R$-modules such that $ f\circ p_M=p_N \circ f$. Denote $\Hom_{(R, P)}(M, N)$ for the set of all $(R,P)$-module homomorphisms, which is a $\bfk$-submodule of $\Hom_{R}(M,N)$, the $\bfk$-module of all $R$-module homomorphisms.
\item For an $(R,P)$-module $(M,p)$, an {\bf $(R,P)$-submodule} is an $R$-submodule $N$ of $M$ such that $p(N)\subseteq N$. Thus $(N, p|_{N})$ is also an $(R,P)$-module.
\end{enumerate}
\end{defn}

\begin{remark}
This definition of Rota-Baxter module is consistent with the Eilenberg's approach to the definition of module, namely the semidirect sum $(R\oplus M, P+p)$ is a Rota-Baxter algebra. Moreover, its quotient by the Rota-Baxter ideal $(M,p)$ is isomorphic to the initial Rota-Baxter algebra $(R,P)$.
\end{remark}

We also remark that Eq.~\eqref{eq:rbm} is compatible with the Rota-Baxter equation \eqref{eq:rba}, i.e.,
\begin{equation}\label{eq:associativity}
P(a)(P(b)p(x))=(P(a)P(b))p(x)=P(P(a)b+aP(b)+\lambda ab)p(x)
\end{equation}
for all $a,b \in R$ and $ x\in M$.  The verification of this is the same as verifying that Eq~\eqref{eq:rba} is compatible with the associativity in $R$, in the sense that applying the associativity $(P(a)P(b))P(c)=P(a)(P(b)P(c))$ to the left hand side of Eq.~\eqref{eq:rba} leads to the identical expression on the right hand side.

For a Rota-Baxter module homomorphism $f$, it is straightforward to verify that the $R$-modules $\ker(f)$, $\im(f)$ and $\cok(f)$ are $(R,P)$-submodules with the obvious operators induced from $p$. The category $(R,P)\Mod$ of $(R,P)$-modules is an abelian category. There is a  forgetful functor $(R,P)\Mod\rightarrow R\Mod$ forgetting the operator $p$, which is exact and faithful.

Given an $R$-module $M$ there could be many $\bf k$-linear operators $p$ making $(M,p)$ a $(R,P)$-module. Let $\on{RB}_P(M)\subseteq \End_{\bf k}(M)$, or simply $\on{RB}(M)$, denote the set of all such operators.
The $R$-module automorphism group $\on{Aut}_R(M)$ acts on the set $\on{RB}_P(M)$ by conjugations.  Two $\bf k$-linear operators $p$ and $p'$ on $M$ define isomorphic $(R,P)$-modules if and only if they are in the same orbit of the action. Depending on the ring $R$ and the $R$-module $M$, $\on{Aut}_R(M)$ is an algebraic group and $\on{RB}_{P}(M)$ is an algebraic variety. One of the question is to describe the moduli space of the isomorphism classes of $(R, P)$-module structures on $M$ in terms of the algebraic group  $\on{Aut}_R(M)$ on $\on{RB}(M)$. When $ R=\bfk((t))=\bfk[[t]]\oplus t^{-1}\bfk[t^{-1}]$ is the Laurent series field with $P$ being the projection to $\bfk[[t]]$, and $M$ being finite dimensional (over $R$), the moduli spaces are studied in \cite{LQ} and are closed related to the affine Grassmannian corresponding the $t$-adic group $ GL_n(\bfk[[t]])$.

Similarly, one define {\bf right $(R,P)$-modules}. In particular, $(R, P)$ is a left (resp. right) $(R,P)$-module under the left (resp. right) multiplication.

A left, right or two sided ideal $I$ of $R$ is called a left, right or two sided {\bf Rota-Baxter ideal}, respectively, if $P(I)\subseteq I$.
As in the case of usual module theory, any left or right Rota-Baxter ideal $I$ of $(R,P)$ is a left or right Rota-Baxter $(R,P)$-module under the restriction $P: I\rightarrow I$.

We remark that any $R$-module $M$ automatically defines an $(R,P)$-module $(M, 0)$. This defines a full  subcategory  $R\Mod$ of $(R,P)\Mod$ whose composition with the forgetful functor $(R,P)\Mod\rightarrow R\Mod$ is the identity.  We will see this from the perspective of the ring of Rota-Baxter operators in Corollary~\mref{co:subcat}.

\subsection{Regular-singular decompositions of Rota-Baxter modules}
\mlabel{ss:decomp}

A $\bfk$-linear operator $p$ on a module $M$ is called {\bf quasi-idempotent of weight $0\neq \lambda\in \bfk$} if $p^2+\lambda p=0$. The usual concept of an idempotent operator is the special case when $\lambda=-1$.
For $\mu\in \bfk$, let
\begin{equation}
M_\mu\colon = \{x\in M\,|\, p(x)=\mu x\}
\mlabel{eq:eigen}
\end{equation}
denote the eigenspace of $M$ for the eigenvalue $\mu$.
A Rota-Baxter operator $P$ of weight $\lambda$ in a $\bfk$-algebra $R$  is called {\bf quasi-idempotent}~\cite{AM} if $P^2+\lambda P=0$.

\begin{prop}
Let $\lambda\in \bfk$ be given.
\begin{enumerate}
\item
Assume that $\lambda$ is invertible. A linear operator $p$ on a $\bfk$-module $M$ is quasi-idempotent of weight $\lambda$ if and only if there is a $\bfk$-module  $M=M'\oplus M''$ such that $p$ is the $-\lambda$ multiple of the projection from $M$ to $M'$ along $M''$:
\begin{equation}
p: M=M'\oplus M''\to M, \quad x=x'+x'' \mapsto -\lambda x'
\text{ for all }  x'\in M', x''\in M''.
\mlabel{eq:rbdecomp}
\end{equation}
If either of the two equivalent conditions holds, then $M'=M_{-\lambda}$ and $M''=M_0$.
\mlabel{it:rbdec0}
\item
Let $(R,P)$ be a Rota-Baxter algebra and $(M,p)$ a $(R,P)$-module. The Rota-Baxter operator $p$ is quasi-idempotent if and only if $p$ is $P(1)$-invariant in the sense that $p(P(1)x)=P(1)p(x)$ for all $x\in M$.
In the case when $(M,p)=(R,P)$, $P$ is quasi-idempotent if and only if $P$ is right $P(1)$-invariant in the sense that $P(uP(1))=P(u)P(1)$ for all $u\in R$.
\mlabel{it:rbdec1}
\item
Assume that $\lambda$ is invertible. A linear operator $P$ on an algebra $R$ is a quasi-idempotent Rota-Baxter operator of weight $\lambda$ if and only if there is a $\bfk$-module decomposition $R=R'\oplus R''$ of $R$ into nonunitary subalgebras $R'$ and $R''$ such that $P$ is the $-\lambda$ multiple of the projection of $R$ to $R'$ along $R''$ as in Eq.~\eqref{eq:rbdecomp}\footnote{So $P$ is called a splitting Rota-Baxter operator in~\cite{BGP}}.
If either of the two equivalent conditions holds, then $R'=R_{-\lambda}$ and $R''=R_0$.
\mlabel{it:rbdec2}
\item
Assume that $\lambda$ is invertible. Let $P:R\to R$ be a quasi-idempotent Rota-Baxter operator of weight $\lambda$ and $R=R_{-\lambda}\oplus R_0$ be as in Item~$($\mref{it:rbdec2}$)$. Let $M$ be an $R$-module and $p:M\to M$ an quasi-idempotent $\bfk$-linear operator. Then $(M,p)$ is a $(R,P)$-module if and only the eigenspace $M_{-\lambda}$ (resp. $M_0$) from Item~(\mref{it:rbdec0}) is a $R_{-\lambda}$-module (resp. $R_0$-module).
\mlabel{it:rbdec3}
\end{enumerate}
\mlabel{pp:rbdecomp}
\end{prop}

\begin{proof}
(\mref{it:rbdec0}) This is a standard linear algebra exercise using the minimal polynomial $ p(p+\lambda)=0$ and the fact that $\lambda$ is invertible to get $ M=M_{-\lambda}\oplus M_0$.
\smallskip

\noindent
(\mref{it:rbdec1})
The first equivalence follows from
$$
P(1)p(x)= p(P(1)x)+p(p(x))+\lambda p(x) \quad \text{ for all } x\in M.$$
The second equivalence follows from
$$P(u)P(1)=P(P(u))+P(uP(1))+\lambda P(u) \quad \text{ for all } u\in R.$$

\noindent
(\mref{it:rbdec2}) This fact is already proved in~\mcite{BGP}.
\smallskip

\noindent
(\mref{it:rbdec3})
Let $(R,P)$ and $(M,p)$ be as given.  Then by Item~(\mref{it:rbdec0}), $M=M_{-\lambda}\oplus M_0$ and $p$ is the $-\lambda$ multiple of the projection from $M$ to $M_{-\lambda}$ along $M_0$. Suppose $(M,p)$ is a $(R,P)$-module. By Eq.~\eqref{eq:rbm},  $M_{-\lambda} = p(M)$ is a $R_{-\lambda}(=P(R))$-module. Further, let $r\in R_0=\ker P$ and $x\in M_0=\ker p$. Then Eq.~\eqref{eq:rbm} gives
$$\lambda p(rx)=P(r)p(x) - p(P(r)x)-p(rp(x))=0.$$
Thus $M_0$ is a $R_0$-module.

Conversely, suppose the eigenspace $M_{-\lambda}$ (resp. $M_0$) from Item~(\mref{it:rbdec0}) is a $R_{-\lambda}$-module (resp. $R_0$-module). We consider two cases in verifying Eq.~\eqref{eq:rbm}.

First consider $u\in R_{-\lambda}$. Then we have
$$ P(u)p(x)= -\lambda up(x) = \left\{ \begin{array}{ll} 0, & v\in M_0, \\
\lambda^2 ux, & v\in M_{-\lambda}. \end{array} \right .
$$
On the other hand, $P(u)=-\lambda u$ gives
$$ p(up(x))+p(P(u)x)+\lambda p(ux) =p(up(x)) =
\left \{\begin{array}{ll}
0, & x\in M_0, \\
p(u(-\lambda)x) = -\lambda p(ux) =-\lambda (-\lambda)ux, & x\in M_{-\lambda}. \end{array} \right .
$$
Here in the last case we applied the assumption that $M_{-\lambda}$ is a $R_{-\lambda}$-module. Thus Eq.~\eqref{eq:rbm} holds in this case.

Next consider $u\in R_0$. Then we have
$ P(u)p(x)=0$. On the other hand,
$$ p(up(x))+p(P(u)x)+\lambda p(ux) = p(up(x))+\lambda p(ux)
=\left\{\begin{array}{ll} 0, & x\in M_0, \\
p(u(-\lambda x))+\lambda p(ux) =0, & x\in M_{-\lambda}. \end{array} \right .
$$
Here in the first case, we have applied the property that $M_0$ is a $R_0$-module. Thus $(M,p)$ is a $(R,P)$-module.
\end{proof}

As an immediate consequence of Proposition~\mref{pp:rbdecomp}.(\mref{it:rbdec1}), as well as \cite[Lemma~1]{BGP}, we have

\begin{coro}
Let $P:R\to R$ be a Rota-Baxter operator satisfying $P(1)\in \bfk$. Then $P$ is quasi-idempotent. Further for any $(R,P)$-module $(M,p)$, $p$ is also quasi-idempotent.
\mlabel{co:rbdec}
\end{coro}

The decomposition $M=M_{-\lambda} \oplus M_0$ in the proposition will be called the {\bf regular-singular decomposition}, motivated by the following example. A more detailed study in this case can be found in~\mcite{LQ}.

\begin{exam} Let $X$ be a complex manifold and $x_0\in X$ be a fixed point. Let $ \calO_{x_0}$ and $\calM_{x_0}$ be the stalks at $ x_0$, of the sheaves $\calO$ and $\calM$ of holomorphic functions  and meromorphic functions respectively.  We know that $\calO_{x_0}$ is a $\CC$-subalgebra of $\calM_{x_0}$. Any linear map $P: \calM_{x_0}\rightarrow \calO_{x_0}\subseteq \calM_{x_0}$ satisfying the Rota-Baxter relation of weight $-1$ such that $ P|_{\calO_{x_0}}=\on{Id_{\calO_{x_0}}}$ defines a regular-singular decomposition $\calM_{x_0}=(\calM_{x_0})_0\oplus \calO_{x_0}$.  One can define an integration theory by taking $\int f:=\int_{x_0}^xP(f)dz $ for all $f\in \calM_{x_0}$.  Thus we can regard the Rota-Baxter algebra $(\calM_{x_0}, P)$ as renormalization in taking a function $f\in \calM_{X_0}$ to get a regular function $ P(f)\in \calO_{x_0}$.
For a sheaf $\calF$ of $\calM$-module, the stalk $\calF_{x_0}$ is an $\calM_{x_0}$-module. The set of sections with singularity at $x_0$  of the sheaf $\calF$   is not a vector subspace.  A Rota-Baxter $(\calM_{x_0}, P)$-module structure $p: \calF_{x_0}\rightarrow \calF_{x_0}$, defines a $\CC$-vector space decomposition $  \calF_{x_0}= (\calF_{x_0})_{-1}\oplus  (\calF_{x_0})_{0}$ with $(\calF_{x_0})_{-1}$ the stalk of sheaf of regular sections and $(\calF_{x_0})_{0}$ be the stalk of sheaf of singular sections.
\mlabel{ex:var}
 \end{exam}

 \begin{exam} Let $C$ be a smooth complex curve and $x_0\in C$.  Let $ \calO_{x_0}$ be the complete  local ring of holomorphic functions at $x_0$ and $ \calM_{x_0}$ be the field of quotients.  Each choice of coordinate $z$ defines a Rota-Baxter algebra structure on $R=\calM_{x_0}$ with $ P(R)=\calO_{x_0}$. For each vector bundle $ \calF$ on $C$, each trivialization of $\calF$ at $x_0$ defines an  $(R, P)$-module structure on $ \calF_{x_0}= \calM_{x_0}\otimes _{\calO_{x_0}}\calF$.  A classification of such module structures on the sheaf $ \calF_{x_0}$ is discussed in \cite{LQ}. Vector bundles on curve with  trivialization at a point was studied in \cite{BL} to describe the conformal blocks in mathematical physics.
 \end{exam}

Suppose that $(R,P)$ is a Rota-Baxter algebra of invertible weight $\lambda\in \bfk$ and $ P(1)\in \bfk $. By Proposition~\mref{pp:rbdecomp}, $(R,P)$ is a quasi-idempotent Rota-Baxter algebra and by Corollary~\mref{co:rbdec}, any $(R,P)$-module $(M,p)$ is quasi-idempotent, giving the regular-singular decomposition $M=M_{-\lambda}\oplus M_0$. Let $\RSD(M)$ denote the set of all regular-singular decompositions $M=M_{-\lambda}\oplus M_0$ as $p$ vary. Then the assignment
$$ p\mapsto (p(M), \ker(p))$$
defines a bijection from $\RB(M)$ to $\RSD(M)$.
Further $ \Aut_R(M)$ acts on $\RSD(M)$ by
$ g(M_1, M_0)=(g(M_1), g(M_0))$. Then the bijection from $\RB(M)$ to $\RSD(M)$ is $\Aut_R(M)$-equivariant. Thus the set of isomorphic classes of $(R,P)$-module structures on $M$ is in bijection with the set of $\Aut_R(M)$-orbits in $ \RSD(M)$. In summary,

\begin{prop} Let $(R,P)$ be a Rota-Baxter algebra of invertible weight $\lambda\in \bfk$ and $ P(1)\in \bfk $. Then for any $R$-module $M$, the map $ \RB(M)\rightarrow \RSD(M)$ defined by $(M,p)\mapsto M_{-\lambda}\oplus M_0$ is an $\Aut_R(M)$-equivariant bijection.
\end{prop}

As an application of Proposition~\mref{pp:rbdecomp}, we give a simple example to demonstrate the distinction between modules over an algebra and modules over a Rota-Baxter algebra. Take $R=\bfk$ to be a field and fix $\lambda\in \bfk$. Then $(\bfk,-\lambda)$ is a Rota-Baxter algebra of weight $\lambda$. We determine the category of finite dimensional $(R,P)$-modules. Such a module is necessarily a finite dimensional $\bfk$-module, so is of the form $M=\bfk^n$ for $n\geq 0$.

First let $\lambda\neq 0$. Since trivially $P(1)\in \bfk$, by Proposition~\mref{pp:rbdecomp}, a linear map $p:\bfk^n\to \bfk^n$ defines a $(R,P)$-module structure on $\bfk^n$ means $p$ is diagonalizable over $\bfk$ with eigenvalues $0$ and $-\lambda$. Thus the category of $(R,P)$-modules is semisimple with exactly two irreducible representations $(\bfk,0)$ and $(\bfk,-\lambda)$.

In the ``limit" case of the Rota-Baxter algebra $(R,P)=(\bfk,-\lambda)$ of weight $\lambda$ when $\lambda=0$, a pair $(M,p)$ is a $(R,P)$-module if and only if $p^2=0$. Thus the category of $(R,P)$-modules is not semisimple with irreducible representation $(\bfk,0)$ as in the nonzero weight case. Instead, the category has two indecomposable representations $(\bfk,0)$ and $(\bfk^2,J_2)$ where $J_2$ is a Jordan block of size 2 with eigenvalue $0$.

Through the above discussion, we see that the category of $ (\bfk,-\lambda)$-modules of weight $\lambda$ is equivalent to the category of $A=\bfk[t]/\langle t^2+\lambda t\rangle$-modules.  A special case of such an algebra $A$ is the Hecke algebra of the symmetric group $S_2$ over a field $\bfk$ which is generated by $T$ subject to the condition $ (T-q)(T+1)=0$ with $ q\in\bfk$. In this case we take $ P=T-q$ and $ \lambda=1+q$.

\begin{remark}
There is a subtle point that is worth noting, namely the pair $(\bfk,0)$ is a Rota-Baxter algebra of weight $-1$ as well as of weight $0$ just discussed. Representation of this Rota-Baxter algebra depends on its designated weight. In fact, as a Rota-Baxter algebra of weight $-1$, its category of modules is semisimple by Proposition~\mref{pp:rbdecomp} by the same argument as above.
\mlabel{rk:zero}
\end{remark}

\subsection{Dual modules and derived modules of Rota-Baxter modules}
We study the relationship of the adjoint operator $\tilde{P}$ of $P$ and the Atkinson factorization with the Rota-Baxter modules.

Recall that for any Rota-Baxter algebra $(R, P)$ of weight $\lambda$, the pair $(R, \tilde{P})$, with $\tilde{P}:=-\lambda\on{I}_R-P$, is also a Rota-Baxter algebra of the same weight $\lambda$. In the same way, if $(M, p)$ is an $(R,P)$-module, then $(M, \tilde{p})$ is an $(R, \tilde{P})$-module, where $\tilde{p}:=-\lambda\on{I}_M-p$.
Furthermore, if $f:(M,p_M)\rightarrow (N,p_N)$ is an $(R,P)$-module homomorphism, then the same $R$-module homomorphism $f: M\rightarrow N$ is an $(R, \tilde{P})$-module homomorphism from $(M,\tilde{p}_M)$ to $(N,\tilde{p}_N)$. We thus obtain the following

\begin{lemma} The assignment $ (M,p)\mapsto (M, \tilde{p})$ for $(M,p)\in (R,P)\Mod$ defines a categorical isomorphism $(R,P)\Mod\rightarrow (R, \tilde{P})\Mod$.
\end{lemma}
\begin{proof}
The resulting functor is an isomorphism since $\tilde{\tilde{P}}=P$ and $\tilde{\tilde{p}}=p.$
\end{proof}

When we use $R$ to denote the Rota-Baxter algebra $(R,P)$, we will also simply use $ \tilde{R}$ to denote the Rota-Baxter algebra $(R, \tilde{P})$. Similarly, when we abbreviate $M$ for the module $(M,p)$, then $\tilde{M}$ will denote the module $(M, \tilde{p})$.

Suppose that $ (R,P)$ is a Rota-Baxter algebra of invertible weight $ \lambda$ and $ P(1)\in\bfk$. Then  $\tilde{P}(1)= -\lambda-P(1) \in \bfk$.
In this case, by Proposition~\mref{pp:rbdecomp} we have $ R=R_{-\lambda}\oplus R_0$ and any
$(R,P)$-module is an $R$-module with a decomposition $M=M_{-\lambda}\oplus M_{0}$ such that
$ R_{-\lambda}M_{-\lambda}\subseteq M_{-\lambda}$ and
$ R_0M_{0}\subseteq M_0$. We denote
$ R=\tilde{R}_{-\lambda}\oplus \tilde{R}_{0} $ for the decomposition of $R$ with respect to $\tilde{P}$.  Similarly for any $(R,P)$-module $(M,p)$, the pair $(M,\tilde{p})$ is an $(R, \tilde{P})$-module with decomposition $ M=\tilde{M}_{-\lambda}\oplus \tilde{M}_{0}$. Then we have
$$\tilde{R}_{-\lambda}=R_0, \quad \tilde{R}_{0}=R_{-\lambda}, \quad \tilde{M}_{-\lambda}=M_0, \quad \tilde{M}_{0}=M_{-\lambda}.$$

Recall~\cite[Thm 1.1.17]{Gub} that, for any Rota-Baxter algebra $(R, P)$, there is a new associative multiplication on $R$ defined by
\begin{equation}\label{eq:star_prod} r_1\star_P r_2:=r_1P(r_2)+P(r_1)r_2+\lambda r_1r_2
\end{equation}
making $(R, \star_P)$ into a nonunitary associative $\bfk$-algebra.
We will denote this $\bfk$-algebra by $ R^{(P)}$. Furthermore $(R^{(P)}, P)$
is still a Rota-Baxter algebra of weight $\lambda$ and
$ P: (R^{(P)}, P)\rightarrow (R,P)$ is a homomorphism of
Rota-Baxter algebras.

Now let $(M,p)$ be an $(R, P)$-module. we define a new linear map
$$\star_p: R\otimes M\rightarrow M,  r\otimes x\mapsto P(r)x+rp(x)+\lambda r x \quad
\text{for all } r\in R, x\in M.$$
The same argument as in the proof of
\cite[Thm. 1.1.17]{Gub} shows

\begin{prop} Let $(R,P)$ be a Rota-Baxter algebra and $(M,p)$ a $(R,P)$-module.
\begin{enumerate}
\item
$p: M^{(p)}\rightarrow M$ is $P$-semi-linear, i.e.,
\[ p(r\star_p x)=P(r)p(x) \quad \text{for all } r\in R, x\in M;\]
\item
$(M, \star_p)$ is a nonunitary $R^{(P)}$-module. We will denote this $R^{(P)}$-module by $ M^{(p)}$;
\item
$(M^{(p)}, p)$ is a Rota-Baxter  $(R^{(P)}, P)$-module;
\item
The functor from $(R,P)\Mod$ to $(R^{(P)}, P)\Mod$ defined by
$ (M, p)\mapsto (M^{(p)}, p)$ is exact.
\item
$r\star_{\tilde{P}}s =-r\star_{{P}}s$ and $ r\star_{\tilde{p}}x =-r\star_{{p}}x $ for all $(R,P)$-module $(M,p)$ with $ r, s\in R$ and $ x\in M$.
\end{enumerate}
\end{prop}

Now let $R$ be a $\bfk$-algebra. If  $\lambda\in \bfk$  is torsion free in $R$, the {\bf additive Atkinson factorization}~\mcite{At,Gub} states that Rota-Baxter operators $P$ of weight $ \lambda$ on $R$ are in one-one correspondence with $\bfk$-linear maps $ f: R\rightarrow R\oplus R$ satisfying the following properties:
if $r\in R$ with $ f(r)=(f_1(r), f_2(r))$ then $ f_1(r)+f_2(r)=-\lambda r$ and
$f(R)$ is closed under the following multiplication in $ R\oplus R$
\[ (r_1, r_2)(s_1, s_2)=(r_1s_1, -r_2s_2).\]
In fact, we can take $f_1:=P$ and $f_2:=\tilde{P}$.
The following is a module version of the additive Atkinson factorization.

\begin{theorem} Let $(R,P)$ be a Rota-Baxter algebra of weight $\lambda$ which is torsion free in $R$. Let $f=(f_1, f_2)$ be the pair of $\bfk$-linear maps corresponding to $P$ by the additive Atkinson factorization. Then an $(R,P)$-module $(M,p)$ is equivalent to a pair of $\bfk$-linear maps $ \rho=(\rho_1, \rho_2): M\rightarrow M\oplus M$ such that $\rho_1(x)+\rho_2(x)=-\lambda x $ for all $x\in M$ and $ f(R)\ast\rho(M)\subseteq \rho(M)$, where
\[ \ast: (R\oplus R)\ot (M\oplus M) \to M\oplus M, \quad (r_1, r_2)\ast(x_1, x_2)\colon=(r_1x_1, -r_2x_2) \text{ for all } r_1, r_2 \in R, x_1, x_2 \in M.\]
\end{theorem}

We remark that when $ \lambda$ is torsion free in $R$ and $M$, and $P(1)\in \bfk$, then $\rho: M\rightarrow \rho_1(M)\oplus \rho_2(M)$ is a linear isomorphism that gives $M=M_{-\lambda}\oplus M_0$.  Thus the regular-singular decomposition can be regarded as a special case of the additive Atkinson factorization.

Let $(R, P)$ be a Rota-Baxter algebra of weight $\lambda$. For any $ \alpha \in \bfk$,  by  Eq.~\eqref{eq:rba}, the operator $ \alpha P\in \End_{\bfk}(R)$ is  a Rota-Baxter operator  of weight $ \alpha\lambda$. Thus when $\bfk$ is a field, the problem of  classifying all Rota-Baxter algebra structure (of all weights) on $R$ is reduced to classifying all Rota-Baxter algebra structures of weight $0$ and $-1$ only.

If $(M, p)$ is an $(R, P)$-module, then the definition of Rota-Baxter modules in Eq.~\eqref{eq:rbm} implies that $(M, \alpha p)$ is an $(R, \alpha P)$-module.  If $f: (M, p_M)\rightarrow (N, p_N)$ is a homomorphism of $(R, P)$-module, then $f: (M, \alpha p_M)\rightarrow (N, \alpha p_N)$ is also a $(R, \alpha P)$-module homomorphism.
\begin{prop} For any $ \alpha \in \bfk$, the assignment $ (M, p)\mapsto (M,\alpha p)$ is a faithful functor. In particular, if $ \alpha \in \bfk$ is invertible, then this functor is a category isomorphism between $ (R, P)\Mod$ and  $ (R, \alpha P)\Mod$.
\end{prop}

Thus when $\bfk$ is a field, one can restrict to the cases of $\lambda=0, -1$.

\subsection{Product of Rota-Baxter algebras} \label{sec:RBprod}
Let $(R_1, P_1)$ and $(R_2, P_2)$ be two Rota-Baxter algebras of weight $\lambda$ over $\bfk$. The product $\bfk$-algebra $R=R_1\oplus R_2$ with componentwise multiplication together with $P=P_1\oplus P_2$ is a Rota-Baxter algebra of the same weight $\lambda$ over $\bfk$. For $i=1$ or $2$, the projective map $(R, P)\stackrel{\pi_i}{\rightarrow} (R_i, P_i)$ is a Rota-Baxter algebra homomorphism. In fact $(R, P)$ is the product of $(R_1, P_1)$ and $(R_2, P_2)$ in the category of Rota-Baxter algebras of fixed weight $\lambda$. We note that the embeddings $(R_i, P_i)\rightarrow (R, P)$ are homomorphisms of nonunitary Rota-Baxter algebras.

We now describe the $(R, P)$-modules in terms of those of $(R_i, P_i)$. Writing $e_1=(1, 0)$ and $e_2=(0,1)$ for the central idempotents in $R$, then the composition $ e_jPe_i$ is $0$ if $i\neq j$. Each $ R$-module $M$ has a decomposition $M=M_1\oplus M_2$ with $M_i=e_iM$ being an $R_i$-module. If $ (M, p)$ is an $ (R, P)$-module, then $ (M_i, p_i)$ is an $(R_i, P_i)$-module with $p_i=e_i p e_i$.  Setting $ p_{ij}:=e_{i}pe_{j}: M_j\rightarrow M_i$, then Eq.~\eqref{eq:rbm} implies that,  for $i\neq j$, $r_i\in R_i$, and $ m_j\in M_j$,
\begin{itemize} \item[(a)] $p_{ji}(r_ip_{ij}(m_j))=0$;
\item[(b)] $p_i(r_ip_{ij}(m_j))=P_i(r_i)p_{ij}(m_j)$.
\end{itemize}

Conversely, given any  $ (R_i, P_i)$-modules $(M_i, p_i)$ for $i=1,2$,
we consider the following diagram
\[ \xymatrix{ M_1\ar@/^3pt/[rr]^{p_{21}}&&\ar@/^3pt/[ll]^{p_{12}}M_2}\]
of linear maps. Such a diagram gives an $(R, P)$-module on $(M,p)$ where $M:=M_1\oplus M_2$ and $p:=(p_1+p_{21},p_{12}+p_2)$ if and only if the $\bfk$-linear maps $p_{12}$ and $p_{21}$ satisfy the conditions (a) and (b) above.

In the following we use a simple example to illustrate that determining representations of the product Rota-Baxter algebra $R=R_1\oplus R_2$ is quite non-trivial.

Let $ \bfk$ be a field and let $R_1=R_2=\bfk$ with $P_1=0$ and $ P_2=\on{Id}$. Then $(R_1, P_1)$ and $(R_2, P_2)$ are Rota-Baxter algebras of weight $-1$. Thus by Proposition~\mref{pp:rbdecomp} (see Remark~\mref{rk:zero}), the category of $ (R_i, P_i)$-modules is semisimple with two irreducible modules  $(\bfk, 1)$ and $(\bfk,0)$. Each module is a $\bfk$-vector space of the form $M=M(0)\oplus M(1)$. Here to avoid ambiguity in the subscripts, we use $M(\kappa)$ to denote the eigenspace of $M$ with eigenvalue $\kappa\in \bfk$. We will write $V=\begin{bmatrix}V(0)\\  V(1)\end{bmatrix}$.

In this case, by taking $i=1$ and $j=2$ in (b) above, we have $p_1(p_{12}(m_2))=0$, i.e., $ p_{12}(M_2)\subseteq M_1(0)$.
Taking $i=2$, (b) implies that $p_2(p_{21}(m_1))=p_{21}(m_1)$. Hence $ p_{21}(M_1)\subset M_2(1)$.

Now applying (a) we have $ p_{12}p_{21}=0$ and $ p_{21}p_{12}=0$. Thus we have the following diagram
\[ \xymatrix{ M_1(0)\ar@/^3pt/[drr]^{p_{21}}&&\ar[ll]_{p_{12}} M_2(0)\\
M_1(1)\ar[rr]_{p_{21}}&& \ar@/^3pt/[ull]^{p_{12}} M_2(1)}\]
such that any composition of the arrows is zero.
Each such diagram, regarded as a module on a quiver, contains a submodule of the form
\[ \xymatrix{ M_1(0)\ar@/^3pt/[drr]&&\ar[ll]0\\
0\ar[rr]&& \ar@/^3pt/[ull] M_2(1)}\]
with the quotient being semisimple. The submodule corresponds to exactly the representations of the preprojective algebra of the quiver $A_2$, which is an interesting subject of study \cite{CB}. Bridgeland used representations of the $\ZZ/2$-graded complexes to construct the whole quantum groups \cite{Br}.  In particular the category of $(R, P)$-module is not semisimple.

In general a $(R, P)$-module corresponds to a representation of the quiver  $Q$
\[ \xymatrix{ (2,0)\ar[r]&(1,0)\ar@/^3pt/[r] &\ar@/^3pt/[l](2,1) &\ar[l] (1,1)}.\]
Let $ \bfk Q$ be the path algebra of this quiver and $I$ be the ideal generated by all paths of length at least 2.  Then the category of $(R, P)$-modules is isomorphic to the module category of the algebra $ A=\bfk Q/I$, which has 4 irreducible modules.

\section{The ring of Rota-Baxter operators and Rota-Baxter modules}
\mlabel{sec:rbo}

We introduce the concept of a ring of Rota-Baxter operators and establish its connection with Rota-Baxter modules.
The structure theorem of this ring will be established in Section~\mref{sec:rbost}.

\subsection{Ring of Rota-Baxter operators}
Similar to the ring of differential operators, we construct the ring of Rota-Baxter operators acting on a Rota-Baxter algebra. Then the category of Rota-Baxter modules is equivalent to the category of modules over the ring of Rota-Baxter operators.

Recall that, given $ \bfk$-algebras $ A$ and $B$, the  free product $ \bfk\langle A, B \rangle $ of $A$ and $B$ is the unique $\bfk$-algebra (up to isomorphism) with $\bfk$-algebra homomorphisms  $ \alpha: A\rightarrow \bfk\langle A, B \rangle$ and $ \beta: B\rightarrow  \bfk\langle A, B \rangle $ satisfying the universal property: for any $\bfk$-module $C$ and any $\bfk$-algebra homomorphisms $\phi: A\rightarrow C$ and $ \psi: B\rightarrow C$, there is a unique $\bfk$-algebra homomorphism $\eta: \bfk\langle A, B \rangle\rightarrow C$ such that $ \phi=\eta \circ \alpha$ and $ \psi=\eta\circ\beta$. In fact $ \bfk\langle A, B \rangle $ is the coproduct in the category of associative $\bfk$-algebras.

\begin{defn}
 Let $(R,P)$ be a Rota-Baxter algebra of weight $\lambda$ and $ \bfk[Q]$ be the polynomial algebra with variable $Q$.  The {\bf ring of Rota-Baxter operators on $(R,P)$}, denoted by $U_{RB}(R, P)$, is defined to be  the quotient
\[ U_{RB}(R, P)=\bfk\langle R,\bfk[Q]\rangle /I,\]
where $I$ is the two-sided ideal of $\bfk\langle R, \bfk[Q]\rangle $
\begin{equation}
I=I_{R,Q}=\langle QrQ-P(r)Q+QP(r)+\lambda Q r \, |\,  r\in R \rangle.
\mlabel{eq:RBO}
\end{equation}

We will simply write $U_{RB}(R)$ for $U_{RB}(R, P)$ if $P$ is understood.
\end{defn}

See~\mcite{RR,RGG} for related constructions and applications to boundary value problems.

We will call an associative algebra $A$ together with a specific element $p\in A$ a {\bf pointed associative algebra} and denote it by $(A, p)$. A homomorphism between two pointed associative algebras $f: A=(A, p)\rightarrow (A', p')$ is an associative algebra homomorphism $f:A\rightarrow A'$ such that $f(p)=p'$.  Thus the pair $(U_{RB}(R,P), Q)$ is a pointed associative algebra.

The definition of $U_{RB}(R,P)$ is translated to the following universal property.
\begin{prop} \label{universal} Let $\sigma:R\to \bfk\langle R,\bfk[Q]\rangle \to U_{RB}(R,P)$ be the natural algebra homomorphism. For any pointed associative $\bfk$-algebra $(A, p)$  and any $\bfk$-algebra homomorphism $\phi: R\rightarrow A$ satisfying
\begin{equation}
\phi(P(r))p=p\phi(r)p+p\phi(P(r))+\lambda p\phi(r) \quad \text{ for all } r\in R,
\mlabel{eq:rbouniv}
\end{equation}
there is a unique pointed associative  $\bfk$-algebra homomorphism
$\eta: (U_{RB}(R), Q)\rightarrow (A, p)$ such that
$\phi=\eta\circ \sigma $.
\end{prop}

\begin{proof} Any element $p$ in $A$ together with a $\bfk$-algebra homomorphism $ \phi: R\rightarrow A$ induces a $\bfk$-algebra homomorphism $\bfk \langle R, Q\rangle\rightarrow A$ sending $Q$ to $p$. The condition \eqref{eq:rbouniv} implies that the ideal $I_{R, Q}$ is in the kernel of this $\bfk$-algebra homomorphism. Thus this $\bfk$-algebra homomorphism induces a unique algebra homomorphism from the quotient
$U_{RB}(R)$ to $A$ with the required property.
\end{proof}

Because of this universal property, one may call $U_{RB}(R)$ the
universal enveloping algebra of the Rota-Baxter algebra $(R, P)$.
However, following the analog of calling the smash product
(or the skew polynomial ring)  $R\#\bfk[d]$ the ring of differential
operators for a differential algebra $(R, d)$ in Section~\mref{diff:mod},  we will call $U_{RB}(R, P)$ the {\bf ring of Rota-Baxter operators for $(R,P)$}.

As a consequence of this universal property, the map $\sigma: R\rightarrow U_{RB}(R)$ is injective by taking $A=R$, $p=0$ and $\phi=\operatorname{Id}_R$. Then we get a $\bfk$-algebra homomorphism $\eta: U_{RB}(R)\rightarrow R$ such that $ \operatorname{Id}_R=\eta\circ \sigma$.  In particular $R=U_{RB}(R)/\langle Q\rangle\cong R$ with $\langle Q\rangle$ being the ideal generated by $Q$ in $ U_{RB}(R)$.  Thus we can regard $R$ as a subalgebra of $U_{RB}(R)$ and $ U_{RB}(R)=R\oplus  \langle Q\rangle$ as a $R$-$R$-bimodule. In the next section we will describe the two sided ideal $\langle Q\rangle$ explicitly.

For a Rota-Baxter module $(M,p)$, the $R$-module structure on $M$ defines a $\bfk$-algebra homomorphism $\phi: R\rightarrow \End_{\bfk}(M)$. The $\bfk$-linear map $ p\in \End_{\bfk}(M)$ defines a $\bfk$-algebra homomorphism
$\psi: \bfk[Q]\rightarrow \End_{\bfk}(M)$ with $\psi(Q)=p$. Thus by the universal property mentioned above, there is a unique $\bfk$-algebra homomorphism $\eta:U_{RB}(R)\rightarrow \End_{\bfk}(M)$, which defines a $U_{RB}(R)$-module structure on $M$.
Conversely, for any $U_{RB}(R)$-module $M$, the $ \bfk$-algebra homomorphism $\eta: U_{RB}(R)\rightarrow \End_{\bfk}(M)$ restricts to the subalgebra $R$ to give an  $R$-module structure on $M$. The element $p=\eta(Q)$  defines a $(R,P)$-module structure on $M$ by Eq.~\eqref{eq:rbouniv}. If $ (M, p_M)$ and $(N, p_N)$ are $(R,P)$-modules, an $R$-module homomorphism $ f: M\rightarrow N$ is an $(R,P)$-module homomorphism if and only if $f$ is a $ U_{RB}(R,P)$-module homomorphism. Thus we have

\begin{theorem} An $(R,P)$-module structure on an $R$-module $M$
is exactly a $U_{RB}(R)$-module structure extending the $R$-module structure on $M$. More precisely, the category of $(R,P)$-modules is isomorphic to the category of  $U_{RB}(R)$-modules.
\mlabel{thm:RBM}
\end{theorem}

See Theorem~\mref{thm:rrboconst} for the structure of $U_{RB}(R)$.
Thus we can identify the category $(R,P)$-Mod of $(R,P)$-modules with the category $U_{RB}(R)$-Mod of $U_{RB}(R)$-modules. In particular, the category of $(R,P)$-modules is an abelian category with enough projective objects.

\begin{exam} We revisit the example at the end of Section~\mref{ss:decomp}. Let $ \bfk$ be any commutative ring and $\lambda\in \bfk$, then $ P=-\lambda : \bfk\rightarrow \bfk$ is a Rota-Baxter operator of weight $\lambda$. Then $U_{RB}(\bfk, P)=\bfk[t]/\langle t(t+\lambda)\rangle$. In fact, $U_{RB}(\bfk, P)$ is the Hecke algebra of $S_2$ over $\bfk$ with parameter $q=\lambda-1$.
\end{exam}

\subsection{Categorical properties}
\mlabel{ss:ind}
We now consider some categorical properties of Rota-Baxter modules and the ring of Rota-Baxter operators, beginning with properties of Rota-Baxter modules.

\begin{lemma} \label{lem:tensor-action}
\begin{enumerate}
\item
If $ (M, p)$ is an $ (R, P)$-module, then for any $ \bfk$-module $V$,
$(V\ot_{\bfk} M, 1\ot p)$ is also an $(R, P)$-module;
\mlabel{it:tensmod1}
\item
For each fixed $\bfk$-module $ V$, the assignment
$$T_V: M \mapsto V\otimes_{\bfk} M \quad \text{for all } M\in (R,P)\Mod,$$
defines an endofunctor of $ (R,P)\Mod$;
\mlabel{it:tensmod2}
\item
Further the assignment $$V\mapsto T_V \quad \text{for all } V\in\bfk\Mod,$$
defines a tensor functor
$\bfk\Mod \rightarrow \End((R,P)\Mod)$, where $ \End((R,P)\Mod)$ is the tensor category of all endofunctors  $F: (R,P)\Mod\rightarrow (R,P)\Mod$ with morphisms being natural transformations;
\mlabel{it:tensmod3}
\item
If $ V$ is a unitary $\bfk$-algebra with unit $u: \bfk\rightarrow V$, then the multiplication $ m: V\ot V\rightarrow V$ defines a natural transformation $\mu: T_V\circ T_V\rightarrow T_V$ which is associative and, together with the unit $\eta:T_{\bfk}=\on{Id}\rightarrow T_{V}$, makes $T_V$ into a monad on $(R, P)\Mod$.
\mlabel{it:tensmod4}
\end{enumerate}
\mlabel{lem:tensmod}
\end{lemma}

\begin{proof} Using Theorem~\ref{thm:RBM}, the lemma follows from the standard argument that, for any associative $\bfk$-algebra $A$, the tensor category $\bfk$-Mod of $\bfk$-modules acts on the category $A$-Mod of $A$-modules through the algebra isomorphism $ A\cong \bfk\otimes_{\bfk} A$.  However, we present a proof using the definition of $(R, P)$-modules to highlight the role played by Rota-Baxter operators. In the following $\otimes=\otimes_{\bfk}$ as usual.
\smallskip

\noindent
(\mref{it:tensmod1}) It follows from a simple verification that the standard action of $R$ on $V\ot M$ defined by $ r(v\otimes x)=v\otimes r x$ satisfies Eq.~(\mref{eq:rbm}).

\noindent
(\mref{it:tensmod2}) For an $(R, P)$-module homomorphism $f:M\to M'$ in $\bfk\Mod$, the map $T_V(f):T_V(M)\to T_V(M')$  defined by
$T_V(f)(v\ot x)=v\ot f(x)$ is a homomorphism of $(R,P)$-modules. Thus $T_V$ is a functor.
\smallskip

\noindent
(\mref{it:tensmod3}) follows since $T_{V\ot V'}=T_V\circ T_{V'}$ \cite[pg. 206, Ex. 2]{Ma}.
\smallskip

\noindent
(\mref{it:tensmod4}) Following the standard terminology of a monad~\cite[Chapter~VI]{Ma}, one verifies the equalities of natural transformations
$$ \mu\circ (T_V\mu) = \mu \circ (\mu T_V): T_V^3\to T,$$
$$ \mu\circ T\eta =\mu \circ \eta T=\id_{T_V}:T_V \to T_V,$$
that define a monad on $(R,P)\Mod$.
\end{proof}

 Let $f: (R', P')\rightarrow (R, P)$ be a homomorphism of Rota-Baxter algebras. If $ (M, p)$ is an $(R,P)$-module, then with the action of $R'$ on $M$ defined by
 $$r'v:=f(r')v  \quad \text{ for all }  r'\in R', v\in M,$$
the $(R,P)$-module $(M,p)$ becomes an $(R', P')$-module which we will denote $f^*(M, p)$. We thus obtain a pullback functor $f^*: (R, P)\Mod\rightarrow (R',P')\Mod$ of abelian categories which is exact and faithful.  We will see that the functor $f^*$ admits a left adjoint functor $f_!$ and right adjoint functor $f_*$ with the help of the ring of Rota-Baxter operators.

If  $g: R\rightarrow R$ is a $\bfk$-algebra automorphism, then
  $(R, g^*(P))$ with $g^*(P)=g^{-1}Pg$ is a Rota-Baxter algebra and $g: (R, g^*(P))\rightarrow(R, P))$ is an isomorphism of Rota-Baxter algebras.  If $(M, p)$ is an $(R,P)$-module, then $ g^*(M)$ is an $R$-module as the pullback of $g$ defined by
$  r\cdot x=g(r)x $ for all $r\in R$ and $x\in M$. Then $(g^{*}(M), p)$ is an $(R, g^{*}(P))$-module.  Thus $ g^*$  is an isomorphism between the categories $(R, P)\Mod $ and $(R, g^{*}(P))\Mod$.

Now we turn our attention to the categorical properties of rings of Rota-Baxter operators.

\begin{prop} Let $f: (R, P_R)\rightarrow (R', P_{R'})$ be a homomorphism of Rota-Baxter algebras.
\begin{enumerate}
\item
The map
$$ f_Q: U_{RB}(R)\rightarrow U_{RB}(R'), \quad f_Q(Q)=Q, f_Q(r)=f(r)\quad \text{for all } r\in R,  $$
is a homomorphism of associative $\bfk$-algebras. Thus we have a functor $U_{RB}: \on{RBA}_{\bfk}\rightarrow \on{Alg}_{\bfk}$.
\mlabel{it:rbo1}
\item
If $ f: (R, P_R)\rightarrow (R', P_{R'})$ is surjective, then so is $f_Q$.
\mlabel{it:rbo2}
\end{enumerate}
\mlabel{pp:rbo}
\end{prop}

\begin{proof}
(\mref{it:rbo1}) Let $\sigma': R'\to \bfk\langle R',\bfk[Q]\rangle \to U_{RB}(R')$ be the natural algebra homomorphism. Then the map
$ R\stackrel{f}{\rightarrow}R'\stackrel{\sigma'}{\rightarrow} U_{RB}(R')$  satisfies the condition
\[ \sigma'(f(r))Q=Q\sigma'(f(r))Q+Q\sigma'(f(P(r)))+\lambda Q\sigma'(f(r)) \quad \text{ for all } r\in R.\]
Then the universal property of $U_{RB}(R)$ gives rise to the $\bfk$-algebra homomorphism $f_Q$.

\noindent
(\mref{it:rbo2})
If $f$ is surjective, then the induced map $F:\bfk\langle R,Q\rangle \to \bfk\langle R',Q\rangle$ and hence its composition with the quotient map $\bfk\langle R',Q\rangle\to U_{RB}(R')$ are surjective. Then the induced map $f_Q$ is also surjective.
\end{proof}

Taking the algebra homomorphism $ f_Q: U_{RB}(R)\rightarrow U_{RB}(R')$ in Proposition~\mref{pp:rbo}, we obtain the pullback functor $ f_Q^*: U_{RB}(R')\Mod\rightarrow U_{RB}(R)\Mod$. In particular, when $(R, P)$ is a Rota-Baxter subalgebra of $ (R', P')$, the functor is the restriction functor. We will simply write $f^*$ for the functor $f_Q^*$. This is consistent with the notation $f^* $ defined above.

Similarly, the functor $f^*$ has a left adjoint functor $$f_!: (R,P)\Mod\rightarrow (R',P')\Mod$$
defined by $f_{!}(M)=U_{RB}(R')\otimes_{U_{RB}(R)}M$ for $M$ in $(R,P)\Mod$, i.e.,
for any $M$ in $(R',P')\Mod$,  the pre-composition with the map $$M\rightarrow U_{RB}(R')\otimes_{U_{RB}(R)}M,  x\to 1\otimes x,$$
defines an isomorphism of $\bfk$-modules
\[ \Hom_{(R',P')}(f_!(M), M)\cong \Hom_{(R,P)}(M, f^*(M)). \]

  There is also a right adjoint functor $f_{*}: (R,P)\Mod\rightarrow (R',P')\Mod$ defined by $f_{*}(M)=\Hom_{U_{RB}(R)}(U_{RB}(R'), M)$.

In case $ (R,P)$ is a Rota-Baxter subalgebra of $(R', P')$ with $f$ being the embedding, we will call $f_!$ the coinduction functor, denote by $\on{CoInd}_{R}^{R'}$. The $(R',P)$-module $\on{CoInd}_R^{R'}M=U_{RB}(R')\otimes_{U_{RB}(R)}M$ is called the coinduced module and
  $\on{Ind}_R^{R'}M=\Hom_{U_{RB}(R)}(U_{RB}(R'), M)$ is called the induced module.

We next discuss some bimodule properties of Rota-Baxter modules.

Let $ (R, P)$ and $(S, P')$ be two Rota-Baxter algebras of weights $ \lambda$ and  $\mu$ respectively. Let $ M$ be an $R$-$S$-bimodule, i.e., $M$ is a left $R$-module and right $S$-module such that $ r(ms)=(rm)s$ for all $ r\in R,\; m \in M, \; s\in S$. A $\bfk$-linear map $ p: M\rightarrow M$ is said to give a {\bf strict  Rota-Baxter bimodule structure} if $ (M, p)$ is a left $(R, P)$-module and also a right $ (S, P')$-module, i.e., for all $ r\in R,\; m\in M, \; s\in S$
\begin{eqnarray} \label{eq:rbbimod1}p(m)P'(s)&=& p(p(m)s+mP'(s)+\mu ms), \\
\label{eq:rbbimod2} P(r)p(m)&=&p(rp(m)+P(r)m+\lambda  rm).
\end{eqnarray}

\begin{lemma} If $ (M, p)$ is a strict $(R, P)$-$(S, P')$-bimodule, then $ (\lambda-\mu) p(rp(m)s)=0$ for all $ r\in R,\; m \in M ,\; s \in S$. In particular, $(\lambda-\mu)p^2(m)=0$ for all $m\in M$.
\end{lemma}
This suggests that the interesting case to consider is when $\lambda=\mu$.

\begin{proof} The identity follows from simplifying  the identity
$(P(r)p(m))P'(s)=P(r)(p(m)P'(s))$
by Eqs.~(\mref{eq:rbbimod1}) and (\mref{eq:rbbimod2}).
\end{proof}

Since $ (M,p)$ is a left $ (R,P)$ -module, it is a left $U_{RB}(R)$-module. Similarly, it is a right $U_{RB}(S)$-module. If we use $Q' \in U_{RB}(S)$ to denote the generator, then we have two $\bfk$-algebra homomorphisms
$U_{RB}(R)\rightarrow \End_{\bfk}(M)$ and $ U_{RB}(S)\rightarrow \End_{\bfk}(M)$ with both $ Q$ and $Q'$ sent to $p$.
\begin{prop} $(M, p)$ is an $(R, P)$-$(S, P')$-bimodule if and only if it is a
$U_{RB}(R)$-$U_{RB}(S)$-bimodule with $ Qm=mQ'$ for $m\in M$.
\end{prop}

We remark that the definition of a strict bimodule requires that both left
and right module structure share the same operator.
 In terms of  the bimodule for the rings of Rota-Baxter algebras,
 this means that the actions of both $Q$ and $Q'$ on the module are the same as indicated in the proposition. One could require two possibly different
 commuting operators $p_l, p_r: M\rightarrow M$ for the left and right
 module structures respectively. In this case, we simply say that
 $ (M, p_l, p_r)$ is a bimodule.
 This  is  a $U_{RB}(R, P_R)$-$U_{RB}(S, P_S)$-bimodule.
 For example, $U_{RB}(R, P)$ is a left and right $(R, P)$-module with
 $p_l$ and $p_r$ simply being the left and right multiplication of the
 element $Q$. As $Q$ needs not be in the center of $U_{RB}(R, P)$,
 the left multiplication and right multiplication by $Q$ are two
 different operators on $U_{RB}(R, P)$. This bimodule may not be strict.

The following proposition is just a consequence of standard properties of
modules over associative algebras.

\begin{prop} Let $M$ be an $(R, P_R)$-$(S,P_S)$-bimodule and let $N$ be a left $(S,P_S)$-module. Then $ M\otimes_{U_{RB}(S,P_S)}N$ is a left
$ (R, P_R)$-module and $N\mapsto M\otimes_{U_{RB}(S,P_S)}N$
defines a functor $ (S, P_S)\Mod \rightarrow (R, P_R)\Mod$.  Similarly,
if $ L$ is a left $(R, P_R)$-module, then $ \Hom_{(R, P_R)}(M, L)$ is
a left $ (S, P)$-module and there is natural isomorphism  of
$\bfk$-modules
\[\Hom_{(R, P_R)}(M \otimes_{U_{RB}(S, P_S)} N, L)\cong \Hom_{(S, P_S)}(N, \Hom_{(R, P_R)}(M, L)).
\]
\end{prop}

We end this section with a followup remark on products of Rota-Baxter modules in Section~\mref{sec:RBprod}.
\begin{remark}
Let $(R_1, P_1)$ and $(R_2, P_2)$ be two Rota-Baxter algebras of weight $ \lambda$. In Section~\ref{sec:RBprod}, we constructed the product Rota-Baxter algebra $(R, P)=(R_1\oplus R_2,P_1\oplus P_2)$. For $i=1$ or 2, the projection map $\pi_i: R\rightarrow R_i, i=1, 2,$ is a homomorphism of Rota-Baxter algebras and thus induces homomorphism of associative algebras
$U_{RB}(\pi_i):U_{RB}(R,P)\rightarrow U_{RB}(R_i,P_i)$. Hence we have a homomorphism $ \pi: U_{RB}(R,P)\rightarrow U_{RB}(R_1, P_1)\times U_{RB}(R_2, P_2)$.  If we use $ Q_i$ to denote the variable $Q$ in $ U_{RB}(R_i, P_i)$, then $ \pi(Q)=(Q_1, Q_2)\in U_{RB}(R_1, P_1)\times U_{RB}(R_2, P_2)$. This homomorphism is not an isomorphism as we have seen in terms of representation theory in Section~\ref{sec:RBprod}. It is not obvious from the definition that $\pi$ is surjective.
 We will see in Remark~\ref{rk:RBprod-ker}  that $ \pi$ is surjective and the kernel will be explicitly constructed.
\mlabel{rk:RBprodring}
\end{remark}

\section{Construction of the ring of Rota-Baxter operators}
\mlabel{sec:rbost}

By Theorem~\mref{thm:RBM}, the study of Rota-Baxter modules is reduced to the study of modules over $U\rbring{R}$, the ring of Rota-Baxter operators. Thus it is necessary to get precise information on the algebra
$U\rbring{R}$. So in this section we provide a general construction of $U\rbring{R}$ and then consider some special cases.

\subsection{The general construction}

We first realize the ring $U\rbring{R}$ of Rota-Baxter operators on a Rota-Baxter algebra $(R,P)$ as an $R$-bimodule.  Recall that $ \tilde{P}=-\lambda I_R-P$.
We note that the relation in Eq.~(\mref{eq:RBO}):
$$ QfQ-P(f)Q+QP(f) +\lambda Q f=0$$
can be regarded as the rewriting rule in the context of rewriting systems~\mcite{BN}
$$ QfQ\mapsto  P(f)Q-QP(f) -\lambda Q f=P(f)Q+Q\tilde{P}(f)$$
that replaces a monomial with multiple $Q$-factors by a linear combination of monomials with fewer $Q$-factors. Iterations of this process lead to a linear combination of monomials with only one $Q$-factor. Thus letting $\langle Q \rangle $ denote the two-sided ideal generated by $Q$ in $U\rbring{R}$, we have
\begin{lemma} \label{lem:ideal}  $\langle Q\rangle =RQR \subseteq U_{RB}(R)$.
\end{lemma}
We next determine the multiplication on $RQR$, characterized by the following lemma.

\begin{lemma} Let $(R, P)$ be a Rota-Baxter algebra. Define a multiplication $\cdot$ on $R\otimes R$ by
\begin{equation}
 (r_1\otimes s_{1})\cdot (r_2\otimes s_2)\colon =r_1P(s_1r_2)\otimes s_2+r_1\otimes \tilde{P}(s_1r_2)s_2.
\mlabel{eq:ast1}
\end{equation}
Then $\cdot$ defines a nonunitary associative algebra structure on $R\ot R$.
\mlabel{lem:mult}
\end{lemma}
\begin{proof}
To check the associativity, applying the easily verified identity $P(r)P(s)+P(\tilde{P}(r)s)=P(rP(s))$, we have
\begin{eqnarray*}
&& ((r_1\ot s_1)\cdot(r_2\ot s_2))\cdot(r_3\ot s_3)\\
&&\quad =( r_1P(s_1r_2)\otimes s_2+r_1\otimes \tilde{P}(s_1r_2)s_2)\cdot(r_3\ot s_3) \\
&& \quad = r_1P(s_1r_2)P(s_2r_3)\otimes s_3+ r_1P(s_1r_2)\otimes \tilde{P}(s_2r_3)s_3    \\
 &&\quad \quad  + r_1P(\tilde{P}(s_1r_2)s_2r_3))\otimes s_3+r_1\ot\tilde{P}(\tilde{P}(s_1r_2)s_2r_3)s_3\\
 &&  \quad=r_1P(s_1r_2{P}(s_2r_3))\otimes s_3+ r_1P(s_1r_2)\otimes \tilde{P}(s_2r_3)s_3+r_1\ot\tilde{P}(\tilde{P}(s_1r_2)s_2r_3)s_3.
 \end{eqnarray*}
This agrees with
 $(r_1\ot s_1)\cdot((r_2\ot s_2)\cdot(r_3\ot s_3))$ by similarly
 applying the identity $\tilde{P}(r)\tilde{P}(s)+\tilde{P}(r{P}(s))=\tilde{P}(\tilde{P}(r)s)$.
\end{proof}

Let $ \alpha: R\otimes R\rightarrow R\otimes R$ be defined
by $ (s_1\otimes r_2)\mapsto P(s_1r_2)\otimes 1 +1\otimes \tilde{P}(s_1r_2)$. Then Lemma~\mref{lem:mult} shows that the diagram
\[ \xymatrix{ R\otimes (R\otimes R)\otimes R\ar[d]_{=} \ar[rr]^{I_R\ot \alpha\ot I_R}&&  R\otimes R\otimes R\otimes R \ar[d]^{m\otimes m} \\
(R\otimes R)\otimes (R\otimes R)\ar[rr]^{\cdot}&&
R\otimes R}\]
commutes. Thus the new multiplication $ \cdot $ is an $R$-$R$-bimodule homomorphism with the standard $R$-$R$-module structure on $R\otimes R$.
This property, together with the balance relation $(tr)\cdot t'=t\cdot (rt')$, implies that the associative multiplication $\cdot$ extends to an associative ring structure on $ R\oplus (R\otimes R)$.

\begin{theorem}
\mlabel{thm:rrboconst}
Let $(R,P)$ be a Rota-Baxter algebra of weight $\lambda$. Then we have an algebra isomorphism
$$U\rbring{R} \cong R \oplus (R\ot R).$$
\end{theorem}
\begin{proof}
Let $\frakS=R\oplus R\ot R$ denote the $\bfk$-algebra obtained before the theorem whose multiplication is still denoted by $\cdot$. We note that $\frakS$ is
generated by $R$ and $1\ot 1$ as a $R$-$R$-bimodule. In particular,
$\frakS$ is generated by $R$ and $1\ot 1$ as a $\bfk$-algebra.
There are the natural embedding of $R\rightarrow \frakS$ and the
algebra homomorphism $ \bfk [Q]\rightarrow \frakS$ given by
$Q\mapsto 1\otimes 1$. Thus by the definition of free products,
there is a unique algebra surjection
$\bfk\langle R, \bfk[ Q]\rangle \rightarrow \frakS$.
Furthermore, for any $r\in R$, we have in $ \frakS$
\[(1\otimes 1)\cdot r\cdot (1\otimes 1)=(1\otimes r)\cdot (1\otimes 1)=P(r)\otimes 1+1\otimes \tilde{P}(r) = P(r)\cdot(1\otimes 1)+(1\otimes 1)\cdot\tilde{P}(r).\]
Thus by Proposition~\ref{universal}  we get a unique
algebra homomorphism
\begin{equation}
\eta: U\rbring{R} \to \frakS
\mlabel{eq:rbsurj2}
\end{equation}
such that $ \eta(\langle Q \rangle)\subset R\otimes R$ which then induces a surjective  $R$-$R$-bimodule map $ RQR=\langle Q \rangle\rightarrow R\otimes R$.  Since $R\otimes R$ is a free  $R\otimes_\bfk R^{op}$-module of rank 1 and $ RQR$ is generated by $Q$ as an $R\otimes R^{op}$-module, the $R$-$R$-bimodule map $\chi:  R\otimes R\rightarrow RQR$ with $1\otimes 1\mapsto Q$ is the inverse of $\eta$ as an $R$-$R$-bimodule map. Hence $ \eta $ is an isomorphism of $\bfk$-algebras.
\end{proof}

Here are some direct consequences of Theorem~\mref{thm:rrboconst}.

\begin{coro} Let $ f: (R, P)\rightarrow (R', P')$ be a homomorphism of Rota-Baxter algebras. The induced map $R\otimes R\stackrel{f\otimes f}{\rightarrow} R'\otimes R'$ together with $f: R\rightarrow R'$ defines the induced algebra homomorphism $\tilde{f}:U_{RB}(R)\rightarrow U_{RB}(R')$.

If $f$ is surjective, then so is $\tilde{f}$. If $f$ is injective, then so is the induced map $\tilde{f}$ provided that $R$ and $R'$ are flat $\bfk$-modules.
\end{coro}

\begin{coro} The projection map $ R\oplus (R\otimes R)\rightarrow R$ is a homomorphism of $\bfk$-algebras. Thus $R\Mod$ is a full subcategory of $(R,P)\Mod$ containing those $(R,P)$-modules $(M,p)$ with $p=0$.
\mlabel{co:subcat}
\end{coro}

Let $(R,P)$ be a Rota-Baxter algebra of weight $\lambda$. Let $R^o$ be the opposite $\bfk$-algebra with multiplication $ r^o s^o:=sr$. Then $(R^{o}, P^{o})$ is also a Rota-Baxter algebra of weight $\lambda$ with $P^o=P$ as $ \bfk$-linear map on $ R^o=R$ as $\bfk$-vector space. In particular, $(\widetilde{R^o}, \widetilde{P^o})=(\tilde{R}^o, \tilde{P}^o)$. Let $U_{RB}(R^o)$ be the ring of Rota-Baxter operators of $(R^o,P^o)$. Then the map
$$U_{RB}(R)=R\oplus (R\otimes R) \rightarrow R^o \oplus (R^o\otimes R^o)=U_{RB}(R^o, \tilde{P}^o)$$ defined by
$r\mapsto r^o$ and $r\otimes s\mapsto s^o\otimes r^o$ is an anti-isomorphism of $\bfk$-algebras. Indeed, we can readily check
\[ (r_1\otimes s_1)\cdot (r_2\otimes s_2)\mapsto (s_2^o\otimes r_2^o)\cdot(s_1^o\otimes r_1^o)\]
\[ r_1P(s_1r_2)\otimes s_2+r_1\otimes \tilde{P}(s_1r_2)\mapsto  s^o_2\otimes P(s_1r_2)^o r_1^o+s^o_2\tilde{P}(s_1r_2)^o \otimes r^o_1 \]
using $s_{1}r_{2}=r_{2}^o s^o_{1}$. Thus we obtain

\begin{prop} For any Rota-Baxter algebra $(R, P)$, the twist map $T: R\otimes R\rightarrow  R\otimes R$, with $T(r_1\ot r_2)=r_2\ot r_1$, induces an algebra isomorphism $U_{RB}(R, P)^o\cong U_{RB}(R^o, \tilde{P}^o)$. In particular, if $ R$ is a commutative $\bfk$-algebra, then $U_{RB}(R, \tilde{P})\cong U_{RB}(R, P)^o$.
\end{prop}

\begin{coro} The category  $\on{Mod-}(R,P)$ of right $(R,P)$-modules is isomorphic to the category $(R^o, \tilde{P}^o)\!\Mod$ of left $(R^o, \tilde{P}^o)$-modules under the standard  isomorphism $\on{Mod-}R\longrightarrow R^o\!\on{-Mod}$ together with sending $ p$ to $ -\lambda -p$ for any right $ (R,P)$-module $(M,p)$.
\end{coro}

\subsection{Special cases and examples}
The results of the last subsection work for any Rota-Baxter algebra. We now consider some special cases.

First as direct consequences of Theorem~\mref{thm:rrboconst}, we display

\begin{coro} If $(R, P)$ is a Rota-Baxter algebra over $\bfk$ which is free over $ \bfk$ with a basis $\{x_{i}\}_{i\in I}$. Then $U_{RB}(R)$ is also $\bfk$-free with a basis $ \{x_{i}\}_{i\in I}\dot{\cup}\{x_{i}\otimes x_{j}\}_{(i,j)\in I\times I}$.  In particular if $R$ is finite dimensional, then $U_{RB}(R) $ has dimension
$(\dim_{\bfk}R)( \dim_{\bfk}R+1)$.
\end{coro}

\begin{coro} If $\bfk$ is a field and $ (R, P)$ is a Rota-Baxter subalgebra of $ (R', P')$, then $ U_{RB}(R)$ is a subalgebra of $U_{RB}(R')$.
\end{coro}

\begin{coro} If $R$ is free over $\bfk$, then $ U_{RB}(R)$ is free as left and right $R$-module.
\end{coro}

The Poincar\'e-Birkhoff-Witt theorem for universal enveloping algebra of a Lie algebra implies that the universal enveloping algebra does not have zero divisors.  As illustrated below, even if $R$ is an integral domain, $U_{RB}(R, P)$ may have zero divisors. The correct analogy of $U_{RB}(R, P)$ should be the restricted enveloping algebra for a restricted Lie algebra over a field of characteristic $p$. They are algebras with operators.

Let $(R, P)$ be a Rota-Baxter algebra with $P=0$.  Then $QrQ=-\lambda Qr$ for all $r\in R$. Thus $ (r_1\otimes s_1)\cdot (r_2\otimes s_2)=-\lambda r_1 \otimes s_1r_2s_2$. In particular, when $\lambda=0$ we have $ U_{RB}(R)=R[t]/\langle t^2\rangle$.
On the other hand, if $P=-\lambda I_R$ is a scalar linear map then $\tilde{P}=0$. Thus $( r_1\otimes s_1)\cdot(r_2\otimes s_2)=-\lambda r_1s_1r_2\otimes s_2$.

We finally consider some special Rota-Baxter algebras.

Note that any algebra $R$ can be realized as a Rota-Baxter algebra
by taking its Rota-Baxter operator to be the identity operator, a
Rota-Baxter operator of weight $-1$. In this case, Eq.~(\mref{eq:ast1})
and its degenerated forms become
$$ (r_1\ot s_1)\cdot(r_2\ot s_2)=
r_1s_1r_2\ot s_2, \quad r_1\cdot(r_2\ot s_2)=r_1r_2\ot s_2,
\quad (r_1\ot s_1)\cdot s_2=r_1\ot s_1s_2.$$

In general, for any $u_1, \cdots, u_k\in R\oplus (R\ot R)$, with either $ u_i \in R$ or $u_i \in R\otimes R$ being pure tensors,  we have
$u_1\cdot \ldots \cdot u_k=w_1\ot w_2$ where $w_1\in R$ is the
product of all factors from $R$ in $u_1,\cdots,u_k$ that appear before the last tensor symbol $\otimes$ while $w_2$ is the product of  the factors from
$R$ after the last tensor symbol $\otimes$, unless all $ u_i \in R$ and there is no $\otimes$ appear.  For example,
$$ (r_1\ot s_1)\cdot r_2 \cdot (r_3\ot s_3)\cdot (r_4\ot s_4) \cdot s_5\cdot s_6 = r_1s_1r_2r_3s_3r_4\ot s_4s_5s_6.$$
It follows that, for any $s_1, s_2\in R$, there is $ (1\otimes r-r\otimes 1)\cdot (s_1\otimes s_2)=0$ even though $r\otimes 1\neq 1\otimes r $ in $ R$ if $ r\not \in \bfk$. Thus $U_{RB}(R)$ has zero divisors even if $R$ is an integral domain.

We next consider the case of divided power Rota-Baxter algebra, given by $(R,P)$ where
$$R=\oplus_{k\geq 0} \bfk u_k,\quad u_m u_n =\binc{m+n}{m}u_{m+n},
m, n\geq 0,$$
and $P(u_k)=u_{k+1}.$
This is a Rota-Baxter algebra of weight zero, in fact the free Rota-Baxter algebra of weight zero on the empty set~\cite{Gub}. Then Eq.~(\mref{eq:ast1}) becomes
\begin{equation}
(u_{m_1}\ot u_{n_1})\cdot (u_{m_2}\ot u_{n_2})=\binc{m_1+n_1+m_2}{m_1,n_1,m_2} u_{m_1+n_1+m_2}\ot u_{n_2}
 \mlabel{eq:div0}
\end{equation}
$$u_{m_1}\cdot (u_{m_2}\ot u_{n_2})=\binc{m_1+m_2}{m_1} u_{m_1+m_2}\ot u_{n_2}, \quad (u_{m_1}\ot u_{n_1})\cdot u_{n_2}=\binc{n_1+n_2}{n_1}u_{m_1}\ot u_{n_1+n_2}.$$
Thus the $\bfk$-algebra $ U_{RB}(R, P)$ has basis $\{ u_i, u_j\otimes u_l \;|\; i, j, l\geq 0\}$ with the above defined multiplication.

We finally consider the case of the Rota-Baxter algebra of Laurent series with the projection to the pole part. By a similar computation, we obtain

\begin{prop}
Let $R=\bfk((t))=\bfk[[t]]\oplus t^{-1}\bfk[t^{-1}]$ be the ring of Laurent series with the Rota-Baxter operator being the projection to the pole part. Then in the construction of $U\rbring{R}$ in Theorem~\mref{thm:rrboconst}, namely 
$
U\rbring{R} = R \oplus (R\ot_\bfk R),
$
the product is given by

$$ (t^i\ot t^j)\cdot(t^k\ot t^\ell)=\left\{\begin{array}{ll}
t^{i+j+k}\ot t^\ell, & j+k<0, \\
t^i\ot t^{j+k+\ell}, & j+k\geq 0.
\end{array} \right . $$
More generally,
for  $\fraka=\sum_{i\geq N} a_i t^i, \frakb=\sum_{j\geq N} b_j t^j, \frakc=\sum_{k\geq N} c_kt^k$ and $ \frakd=\sum_{\ell\geq N} d_\ell t^\ell$ in $\bfk((t))$, we have

$$(\fraka\ot \frakb)\cdot (\frakc\ot \frakd)\\
=
(\fraka  \sum_{j,k\geq N, j+k<0} b_jc_k t^{j+k})\ot  \frakd + \fraka\ot (\sum_{j,k\geq N, j+k\geq 0} b_jc_k t^{j+k}  \frakd).
$$
\mlabel{pp:rblau}
\end{prop}

We again revisit the product Rota-Baxter algebras considered in Section~\mref{sec:RBprod} and Remark~\mref{rk:RBprodring}.
\begin{remark}
Let $(R_1, P_1)$ and $(R_2, P_2)$ be two Rota-Baxter algebras of
the same weight $ \lambda$ and
$ (R, P):=(R_1\oplus R_2, P_1\oplus P_2)$ be the product Rota-Baxter algebra constructed in Section~\ref{sec:RBprod}.
Consider the homomorphism
$$ \pi: U_{RB}(R,P)\rightarrow U_{RB}(R_1, P_1)\times U_{RB}(R_2, P_2)$$
defined in Remark~\ref{rk:RBprodring}. Noting that $R=R_1\oplus R_2$ as a $\bfk$-module. Theorem~\ref{thm:rrboconst} gives the $\bfk$-module decomposition
\[ R\oplus (R\otimes R)=\big(R_1\oplus (R_1\otimes R_1)\big)\oplus \big(R_2\oplus (R_2\otimes R_2)\big)\oplus \big((R_1\otimes R_2)\oplus (R_2\otimes R_1)\big).\]
The map $\pi$ restricted to subspace $\big(R_1\oplus (R_1\otimes R_1)\big)\oplus \big(R_2\oplus (R_2\otimes R_2)\big)$ defines a $\bfk$-linear isomorphism
\[ \big(R_1\oplus (R_1\otimes R_1)\big)\oplus \big(R_2\oplus (R_2\otimes R_2)\big)\stackrel{\pi}{\longrightarrow} U_{RB}(R_1, P_1)\times U_{RB}(R_2, P_2)
.\]
Thus $ \pi$ is onto with kernel $ (R_1\otimes R_2)\oplus (R_2\otimes R_1)$.  The category of $ U_{RB}(R_1, P_1)\times U_{RB}(R_2, P_2)$-modules is $U_{RB}(R_1, P_1)\Mod \times U_{RB}(R_2, P_2)\Mod$, which is a full subcategory of $ U_{RB}(R, P)\Mod$ consisting of all modules on which $ (R_1\otimes R_2)\oplus (R_2\otimes R_1)$ acts as zero. They are exactly those modules with $p_{12}=0=p_{21}$ as described in Section~\ref{sec:RBprod}.
\label{rk:RBprod-ker}
\end{remark}

\section{Endomorphism Rota-Baxter algebras}
\mlabel{sec:mat}

After a general study of representations of Rota-Baxter algebras, we now turn to the matrix representations which motivated our study.

Let $(A,Q)$ be a commutative Rota-Baxter algebra (of weight $\lambda$). For any positive integer $n$, we obtain an operator $\frakQ$ on the algebra of $n\times n$ matrices $M_n(A)$ on $A$ by defining $\frakQ$ entry-wise:
$$ \frakQ([a_{ij}])=\left [ Q(a_{ij})\right ].$$
It is easy to check~\mcite{EGGV,EGs} that $\frakQ$ is a Rota-Baxter operator of weight $\lambda$. Such a Rota-Baxter algebra is called a {\bf matrix Rota-Baxter algebra.}

Let $(R,P)$ be a Rota-Baxter algebra (of weight $\lambda$). A {\bf matrix representation} with coefficients in $A$ of $(R,P)$ is a homomorphism
$$ f: (R,P) \to (M_n(A),\frakQ)$$
of Rota-Baxter algebras, first appearing in renormalization of quantum field theory~\mcite{EGGV,EGs}. We give a general discussion in this section and give an algebraic Birkhoff factorization for Rota-Baxter modules.

\subsection{Endomorphism Rota-Baxter algebras from coalgebras} \label{sec:2.2}

Let $(A,Q)$ be a  Rota-Baxter algebra. Let $M_{m,n}(A)$ be the set of all $m\times n$ matrices with entries  in $A$. It is naturally an $M_{m}(A)$-$M_n(A)$-bimodule.

We define $\frakQ_{m,n}: M_{m,n}(A)\rightarrow M_{m,n}(A)$ by $ \frakQ_{m,n}(r_{ij})=(Q(r_{ij}))$, which is a $\bfk$-linear map.
\begin{lemma} Let $(A,Q)$ be a Rota-Baxter algebra. For any positive integers $\ell, m, n $, and $ X\in M_{\ell,m}(A)$ and $ Y \in M_{m,n}(A)$, we have
\begin{equation}
\frakQ_{\ell,m}(X)\frakQ_{m,n}(Y)=\frakQ_{\ell, n}(\frakQ_{\ell,m}(X)Y+X\frakQ_{m,n}(Y)+\lambda XY). \label{rb-op}
\end{equation}
\end{lemma}
\begin{proof} Considering the $(i, j)$-entry of the left hand side matrix, we have
\begin{eqnarray*}
\sum_{l=1}^{m}Q(x_{il})Q(y_{lj})&=&\sum_{l=1}^{m}Q(Q(x_{il})y_{lj}+x_{il}Q(y_{lj})+\lambda x_{il}y_{lj})\\
&=&\sum_{l=1}^{m}Q(Q(x_{il})y_{lj})+\sum_{l=1}^{m}Q(x_{il}Q(y_{lj}))+\sum_{l=1}^{m}\lambda Q(x_{il}y_{lj})
\end{eqnarray*}
which is exactly the $(i,j)$-entry of the matrix on the right hand side of the equation.
\end{proof}

Taking $\ell=m=n$, one recovers the fact that $(M_n(A), \frakQ_{n,n})$ is a Rota-Baxter algebra.
Furthermore $(M_{m,n}(A),\frakQ_{m,n})$ is a strict $(M_m,\frakQ_{m,m})$-$(M_n,\frakQ_{n,n})$-bimodule in the sense of  Eq.~\eqref{eq:rbbimod1}-\eqref{eq:rbbimod2} More precisely, $M_{m,n}(A)$ is a left $M_{m}(A)$-module, a right $ M_n(A)$-module and the operator $\frakQ_{m,n}$ is compatible with the operators $\frakQ_{m,m}$ and $\frakQ_{n,n}$ using  Eq.~\eqref{rb-op}.

The above construction works more generally in the context of coalgebras~\cite{Sw}. Take a coalgebra $H$ over $\bfk$ with comultiplication $ \Delta: H \rightarrow H\ot H$ and co-unit $\epsilon: H \rightarrow \bfk$.  We recall that a {\bf right comodule} of $H$ is a $\bfk$-module $M$ together with a linear map
$$ \delta: M \to M\ot H$$
such that
$$ (\delta\ot 1)\circ \delta = (1\ot\Delta)\circ \delta.$$
If a $\bfk$-submodule $M$ of $H$ is a right coideal of $H$ in the sense that $\Delta(M)\subseteq M\ot H$, then $M$ is a right comodule of $H$. This is the case considered in physics applications~\mcite{EGGV,EGs}. Then the quotient $\bfk$-module $ H/M$ is also a right $H$-comodule.

For any associative $\bfk$-algebra $A$, the set $ H(A):=\Hom_{\bfk}(H,A)$ is an associative algebra with the convolution product $ (f_1\ast f_2)$ defined by
\[  (f_1\ast f_2)(h)=\sum_{(h)}f_1(h_{(1)})f_2(h_{(2)}) \quad \text{for all } f_1, f_2 \in \Hom_{\bfk}(H,A),\]
using Sweedler's notation of $\Delta(h)=\sum_{(h)}h_{(1)}\ot h_{(2)}$.
In particular, when $H$ is a bialgebra, then the subset $ \Hom_{\bfk\text{-Alg}}(H, A)$ is closed under $\ast$ and becomes a semigroup. If $ A$ is a Rota-Baxter algebra with Rota-Baxter operator $Q$, then $\Hom_\bfk(H,A)$, with the operator $P$ defined by
$$P(f)=Q\circ f \quad \text{for all } f \in \Hom_\bfk(H,A),$$
is also a Rota-Baxter algebra~\cite{EGK3}.

We now consider $M\otimes A$ as a right $A$-module and $\End_{A}(M\otimes A)$ as the endomorphism algebra of the right $A$-module.  For $f\in \Hom_\bfk(H,A)$, define $\phi(f) \in \End_{A}( M\ot A)$ by
\begin{equation}
\phi(f)(m\otimes a)=\sum_{(m)}m_{(0)}\ot f(h_{(1)})a ,
\mlabel{eq:psie}
\end{equation}
where  $ a\in A$ and $ m \in M$ and $\delta(m)=\sum_{(m)}m_{(0)}\ot h_{(1)}$.
Then the map
\begin{equation}
\phi: \Hom_\bfk(H,A) \to \End_A(M\otimes A), \quad f\mapsto \phi(f) \quad \text{for all } f\in \Hom_\bfk(H,A),
\mlabel{eq:psi}
\end{equation}
is a $\bfk$-algebra homomorphism: $\phi(f\ast g)=\phi(g)\circ \phi(f)$, making $M\ot A$ into a right module for the algebra $\Hom_\bfk (H,A)$.

When $(A,Q)$ is a Rota-Baxter algebra, we define an $A$-linear operator $\pprime$ on $\End_\bfk(M\ot A)$ by
$$ \pprime(g)(m\ot a)=(1\ot Q)(g(m\ot 1))a \quad \text{ for all } g\in \End_\bfk(M\ot A),  m\in M,  a\in A.$$
Then the pair $(\End_A(M\ot A), \pprime)$ is a Rota-Baxter algebra of weight $ \lambda$. In fact, for $ g_1, g_2 \in \End_A(M\ot A)$ and $m\in M$, denote
\[ g_2(m\ot 1)=\sum_i m_i \ot a_i, \quad g_1(m_i\ot 1)=\sum_j m_{ij}\ot a_{ji}.\]
Then we have
\[ (g_1\circ g_2)(m\ot 1)=\sum_{i,j}m_{ij}\ot a_{ji}a_i, \quad \pprime(g_2)(m\ot 1)=\sum_i m_i \ot Q(a_i), \]
$$ \pprime(g_1)(m_i\ot 1)=\sum_j m_{ij}\ot Q(a_{ji}),\quad
 (\pprime(g_1\circ g_2))(m\ot 1)=\sum_{i,j}m_{ij}\ot Q(a_{ji}a_i).$$
pUsing these expressions one verifies
\begin{eqnarray*}
(\pprime(g_1)\circ \pprime(g_2))(m\ot 1)&=& \sum_{i,j}m_{ij}\ot Q(a_{ji})Q(a_i) \\
&=&\sum_{i,j}m_{ij}\ot Q(a_{ji}Q(a_i)+Q(a_{ji})a_i+\lambda a_{ji}a_i) \\
&=& \pprime(g_1\circ \pprime(g_2)+\pprime(g_1)\circ g_2 +\lambda g_1\circ g_2)(m\ot1).
\end{eqnarray*}
Thus we are led to the following result.

\begin{theorem} Fix a Rota-Baxter algebra $(A, Q)$ of weight $\lambda$, a coalgebra $ H$, and a right $H$-comodule $M$.
\begin{enumerate}
\item The pairs $(\Hom_\bfk(H,A), \homp)$ and $(\End_A(M\ot A), \pprime)$ are Rota-Baxter algebras;
\mlabel{it:homrba1}
\item The algebra homomorphism  $\phi: \Hom_{\bfk}(H, A)\rightarrow \End_A(M\ot A)$ defined in Eq.~$($\mref{eq:psie}$)$ is a homomorphism of Rota-Baxter algebras;
\mlabel{it:homrba2}
\item Equipped with the $\bfk$-linear operator $ p: M\ot A \rightarrow M\ot A$ defined by $ p(m\ot a)=m\ot Q(a)$, the pair $(M\ot A,p)$ becomes a $(\End_A(M\ot A),\pprime)$-module.
\mlabel{it:homrba3}
\end{enumerate}
\mlabel{thm:homrba}
\end{theorem}

\begin{proof}
(\mref{it:homrba1}) has been proved before the theorem.
\smallskip

\noindent
(\mref{it:homrba2}) For $f\in \Hom_\bfk(H,A)$ and $m\ot a\in M\ot A$, we have
$$ \phi(P(f))(m\ot a)=\sum_{(m)}m_{(0)}\ot P(f)(h_{(1)})a =\sum_{(m)}m_{(0)}\ot Qf(h_{(1)})a$$
and
$$ \pprime(\phi(f))(m\ot a)=(1\ot Q)(\phi(f)(m\ot 1))a =(1\ot Q)\Big(\sum_{(m)}m_{(0)}\ot f(h_{(1)})\Big)a=\sum_{(m)}m_{(0)}\ot Qf(h_{(1)})a.$$
Hence $\phi(P(f))=\pprime(\phi(f))$, as needed.
\smallskip

\noindent
(\mref{it:homrba3})
Denote $g(m\ot 1)=\sum_i m_i\ot a_i$. We have
\begin{eqnarray*}
\pprime(g)p(m\ot a)&=&
\pprime(g)(m\ot Q(a))\\
&=& (1\ot Q)(g(m\ot 1))Q(a)\\
&=& (1\ot Q)\Big(\sum_i m_i\ot a_i\Big)Q(a)\\
&=& \sum_i m_i\ot Q(a_i)Q(a)
= \sum_i m_i\ot \Big( Q(a_iQ(a) + Q(a_i)a +\lambda a_ia\Big)\\
&=& (1\ot Q)\Big( \Big(\sum_i m_i\ot a_i\Big)Q(a)+ (1\ot Q)\Big(\sum_i m_i\ot a_i\Big)a +\lambda \Big(\sum_i m_i\ot a_i\Big)a\Big)\\
&=& p\Big( g p(m\ot a) + \pprime(g) (m\ot a) +\lambda g(m\ot a)\Big),
\end{eqnarray*}
as needed.
\end{proof}

We now make connection with the matrix representation of Rota-Baxter algebras in~\cite{EGGV,EGs}.

Consider the bialgebra  $H=M_n(\bfk)$ with the standard matrix basis $ E_{ij}$ and the comultiplication $ \Delta(E_{ij})=\sum_{l}E_{il}\otimes E_{lj}$. Then the Rota-Baxter algebra structure on $ M_n(A)=\Hom_{\bfk}(M_n(\bfk), A)$ is the same as the one defined in Theorem~\mref{thm:homrba}. If we consider the standard right $M_n(\bfk)$-comodule $ M=\bfk^n$ with standard basis $\{E_1, \cdots, E_n\}$ and $ \delta(E_i)=\sum_{l=1}^n E_l \otimes E_{li}$, then $\bfk^n\otimes A=A^n$ is the right free $A$-module identified with $M_{n,1}(A)$ as a left $M_{n}(A)$-module. Thus a matrix representation of $(R, P)$ over $ (A, Q)$ can be interpreted as a Rota-Baxter algebra homomorphism $ (R,P)\rightarrow (M_n(A), \pprime)$.

There are also other types of matrix Rota-Baxter algebras. Let $ I_{>}=\bfk$-$\on{Span}\{E_{ij}\, |\,i>j\}$, then $ I_>$ is a coideal of $ M_{n}(\bfk)$ and its quotient $M_{n}^u(\bfk):= M_n(\bfk)/I_{>}$ is the upper-triangulate matrices. Then
$M_{n}^u(A)$ is the Rota-Baxter subalgebra of $ M_n(A)$. One can define many other subalgebras of $M_n(A)$ in a similar way. We will describe the representation of these algebras.
In particular, the counit $ \epsilon: H\rightarrow \bfk$ is a homomorphism of coalgebras and thus defines a homomorphism $\epsilon^*: A=\Hom_{\bfk}(\bfk, A)\rightarrow \Hom_\bfk(H, A)$ of Rota-Baxter algebras. Here $ \epsilon^*(a)(h)=\epsilon(h)a$ for all $ h\in H$ and $a\in A$.

In general, let  $M$  be a $H$-comodule which is a free $\bfk$-module with a basis $X$. Then the right $A$-module $M\ot A$ is also free with the same basis $X$.     Fixing a linear order on $X$, then from $\phi$ defined in Eq.~\eqref{eq:psi}  we obtain a $\bfk$-algebra  homomorphism
$$f: \Hom_\bfk(H,A)\to \End_{A}(M\otimes A)=M^c_{|X|\times |X|}(A),$$
where $M^c_{|X|\times |X|}(A)$ denote the matrices with finitely many nonzero entries in each column, giving rise to a matrix representation of $H(A)$ mentioned at the beginning of this section.

We summarize the above discussion as follows.

\begin{prop}
Let $(A,Q)$ be a Rota-Baxter commutative algebra of weight $\lambda$, let $H$ be a coalgebra and let $M$ be a right $H$-comodule. If $M$ is a free $\bfk$-module with basis $X$, then there is a matrix representation
$$f: \Hom_\bfk(H,A)\to M^c_{|X|\times |X|}(A)$$
induced from the algebra homomorphism $\phi$ in Theorem~\mref{thm:homrba}.
\mlabel{thm:maxrb}
\end{prop}

Fixing a linear order on $X$, we define $M^u_{X\times X}(A)$ as the subalgebra of  $M_{X\times X}(A) $ consisting $X\times X$ upper triangular matrices with entries in $A$. It still carries a Rota-Baxter operator acting on a matrix entrywise, giving rise to a matrix Rota-Baxter algebra. If the $H$-comodule structure  $\delta: M \rightarrow M\otimes H$ has the property $ \delta(x)=\sum_{x'} x'\otimes h_{x', x}$ with $ h_{x', x}=0$ unless $ x'\leq x$, then we have $\im\, \phi \subseteq M^u_{X\times X}(A)$.
Such representations have appeared in QFT renormalization~\mcite{EGGV,EGs} as alluded to above.

The definition of $ \phi$ in Eq.~\eqref{eq:psi} can also be extended to $A$-modules. Let $(V, p_V)$ be a left $(A,Q)$-module and $ M$ a right $H$-comodule.
We now define a $\Hom_{\bfk}(H,A)$-module structure on $M\otimes V$ by
\begin{equation}  \label{eq:HAm} f(m\otimes v)=\sum_{m}m_{(0)} \otimes_{\bfk} f(h_{(1)})v \quad \text{ for all } f \in \Hom_{\bfk}(H, A), m\in M, v\in V.
\end{equation}
With the $\bfk$-linear map $1_{M}\otimes_{\bfk} p_{V}: M\otimes V\rightarrow M\otimes V$, we obtain a Rota-Baxter module for $H(A)=\Hom_{\bfk}(H,A)$. Let $H$-Comod denote the category of all right $ H$-comodules.
\begin{prop}
The assignment $(M, (V, p_V))\mapsto (M\otimes V, 1\otimes p_V)$ defines a bifunctor
\[H\text{-Comod} \times (A,Q)\on{-Mod}\rightarrow (H(A),P)\on{-Mod}.\]
\end{prop}

If $H$ is a Hopf algebra, then the category $ H\on{-Comod}$ is a tensor category. If we take $ H=\bfk$ then $ H(A)=A$. In this case the bifunctor is the same as those described in Lemma~\ref{lem:tensor-action}.

As a natural question, if $M$ is a simple $H$-comodule and $V$ is a simple $(A,Q)$-module, will $M\otimes V$ be a simple $H(A)$-module?
Further properties of this functor should be studied relating the categories  $(A, Q)\Mod$ and $(H(A), P)\Mod$.

We also remark that the $H$-comodule structure $M\rightarrow M\otimes H$ plays the role of vector bundles with connections, with $H$ being in the coalgebra of differential forms and the algebra $H(A)$ plays the role of the algebra of differential operators.

\subsection{Bifunctors and schemes from Rota-Baxter algebras}
\label{sec:2.4}
In this subsection we briefly discuss the bifunctor and group schemes on the Hom functor $H(A):=\Hom_\bfk(H,A)$ with the additional structure of a Rota-Baxter structure on either the domain algebra $A$ or the codomain coalgebra $H$.

\subsubsection{Rota-Baxter structure on the codomain}
When the algebra $A$ is equipped with a Rota-Baxter operator $Q$, we obtain a Rota-Baxter operator on $H(A)$ by $\homp(f)=Q\circ f $ for all $ f\in H(A)$. If $\rho:H\rightarrow H'$ is a homomorphism of coalgebras, then the map $ \rho^*: \Hom_\bfk(H', A)\rightarrow \Hom_\bfk(H,A)$ defined by $\rho^*( f)= f\circ \rho$ is a $\bfk$-algebra homomorphism and
\[P(\rho^*(f))=Q\circ (\rho^*(f))=Q\circ f\circ \rho=\rho^*(Q\circ f)=\rho^*(P(f)).\]
Thus $ \rho^*$ is a homomorphism of Rota-Baxter algebras.  In particular, let $I$ be a coideal of $H$ and  $\rho: H\rightarrow H/I$ be the quotient homomorphism of coalgebras. Then
$\rho^*: (H/I)(A)\rightarrow H(A)$ is an embedding of Rota-Baxter algebras.

On the other hand, if $\tau: (A, Q)\rightarrow (A', Q')$ is a homomorphism of Rota-Baxter algebras, then for each coalgebra $H$, the map $ H(\tau): H(A)\rightarrow H(A')$ defined by $ f\mapsto \tau\circ f$ is also a homomorphism of Rota-Baxter algebras.

If we use $\operatorname{RBA}_{\bfk}$ to denote the category of all Rota-Baxter $\bfk$-algebras as above and $ \on{Coalg}_{\bfk}$ to denote the category of $\bfk$-coalgebras, then each Rota-Baxter algebra $(A, Q)$ defines a functor $ \on{Coalg}_{\bfk} \rightarrow \End(\on{RBA}_{\bfk})$, by $ H\mapsto H(A)$, and each homomorphism $ \tau: (A, Q)\rightarrow (A', Q')$ defines a natural transformation $H(A)\mapsto H(A')$. Thus we have a bifunctor
\[\on{Coalg}_{\bfk} \times \on{RBA}_{\bfk}\rightarrow \on{RBA}_{\bfk}, \; (H, A)\mapsto \Hom_{\bfk}(H, A)\]
which is contravariant in the first entry, giving rise to a functor
\begin{equation}
\on{Coalg}^{op}_{\bfk}\rightarrow \End(\on{RBA}_\bfk), \quad H\mapsto H(?).
\mlabel{eq:hfunc}
\end{equation}

Given two coalgebras $H$ and $H'$, since the tensor product coalgebra $H\otimes_{\bfk}H'$ is defined by setting
\[ \Delta_{H\ot H'}=(1\ot T_{23}\ot 1)\circ (\Delta_H\otimes \Delta_{H'}),\]
we have $(H\otimes H')(A)=H(H'(A))$ by using the adjoint property
\[ \Hom_{\bfk}(H\otimes H', A)=\Hom_{\bfk}(H, \Hom_{\bfk}(H', A))\]
 of $\bfk$-modules. In particular, using the isomorphism of coalgebras $ H\otimes_{\bfk} H'\cong H'\otimes_{\bfk}H $, we have $ H(H'(A))\cong H'(H(A))$.

If $H$ is in addition a bialgebra with multiplication $m: H\otimes H\rightarrow H$, then the functor $T_{H}:=\Hom_\bfk(H, -): \on{RBA}_\bfk\rightarrow \on{RBA}_{\bfk}$ is a comonad with $ m^*: T_H\rightarrow T_{H}\circ T_H$ and the counit $u^*: T_H \rightarrow \on{Id}=T_{\bfk}$ defined by the identity $u: \bfk\rightarrow H$. This follows from a similar argument as the monad case in Lemma~\mref{lem:tensmod}.(\mref{it:tensmod4}). Applying~\cite{Ma}, we now summarize the above construction as follows.

\begin{theorem} The functor in Eq.~$($\mref{eq:hfunc}$)$ is a tensor functor.  In particular, monoid objects in  $ \on{Coalg}_{\bfk}$ corresponds to monads on $\on{RBA}_\bfk$.
\end{theorem}

\subsubsection{Rota-Baxter structure on the domain}

We next consider the case when the coalgebra $H$ is equipped with a Rota-Baxter structure and make connection with affine schemes.

Similar to schemes corresponding to an algebraic $\bfk$-variety,
which is a set functor $\bfk$-$\on{Alg}\rightarrow \on{Set}$. An affine
$\bfk$-scheme is a representable functor $ \Hom_{\bfk-\on{Alg}}(H, ?)$
with $H$ in $ \bfk$-$\on{Coalg}$. An affine group scheme is a
representable functor $ \Hom_{\bfk-\on{Alg}}(H, ?)$ with $ H$ being a
commutative Hopf algebra with the group multiplication being the
convolution product.

We now define an affine Rota-Baxter scheme as a functor
$ X: \bfk\on{-Alg}\rightarrow \on{RBA}_{\bfk}$ defined by
\[ A\mapsto  \Hom_{\bfk}(H, A)\]
with $ H$ being a fixed Rota-Baxter coalgebra in the sense of~\mcite{JZ,ML}. So  $H$ is a coalgebra together with a linear map $ \sigma: H\rightarrow H$ satisfying the linear dual of the Rota-Baxter axiom in Eq.~(\mref{eq:rba}):
\begin{equation} \xymatrix{
H \ar[d]_{1_H}\ar[r]^{\sigma}& H\ar[r]^{\Delta}&H\otimes H
\ar[d]^{\sigma\otimes 1+1\otimes \sigma +\lambda 1\otimes 1} \\
H\ar[r]^{\Delta}&H\otimes H\ar[r]^{\sigma\otimes \sigma} & H\otimes H
}
\end{equation}
As in the above references, we have
\begin{prop} Let $H$ be a Rota-Baxter coalgebra. Then for any $\bfk$-algebra $ A$, $R=\Hom_{\bfk}(H, A)$ is an associative algebra with $1$ and an operator $ P: R\rightarrow R$ defined by $P(f)=f\circ \sigma$ such that $(R, P)$ is a Rota-Baxter algebra of weight $ \lambda$. In particular the assignment $ A\mapsto (\Hom_{\bfk}(H, A), P) $ is a covariant functor.
\end{prop}

Following the philosophy of Grothendieck (see \cite{DG} and \cite{Jan}) one can define more general Rota-Baxter functors $ \calX: \bfk\on{-Alg}\rightarrow RBA_{\bfk}$ as an $\bfk$-algebra functor $\calX$ together with a $\bfk$-linear natural transformation $ P: \calX\rightarrow \calX$. Given a Rota-Baxter  $\bfk$-functor $\calX$, an $\calX$-module is a functor $ \calM: \bfk\on{-Alg}\rightarrow {\bfk}\Mod$ such that for each $ A\in \bfk\on{-Alg}$, $ \calM(A)$ is a right $A$-module and a left $ \calX(A)$-module with a $p_A:  \calM(A)\rightarrow  \calM(A)$ which is $A$-linear (as a right $A$-module) and such that $(\calM(A), p_A)$ is an $ (\calX(A), P_A)$-module.

For example, let $M$ be an $H$-comodule and also a Rota-Baxter comodule in the sense that there is a $\bfk$-linear map
$\rho: M\rightarrow M$ making the following diagram commute
\begin{equation}  \xymatrix{ M\ar[r]^{\delta} \ar[d]_{\delta}& M\otimes H\ar[d]^{\rho\otimes 1 + 1\otimes \sigma +\lambda 1 \otimes 1}\\
M\otimes H\ar[r]_{\rho\otimes \sigma}& M\otimes H.}
\end{equation}
For such a pair $ (M, \rho)$, we define $ \calM(A)=M\otimes_{\bfk}A$. Then $ \calM (A)$ is a left $ \Hom_{\bfk}(H, A)$-module by
\[ f\bullet  (m\otimes a):=(1_M\otimes m_A)\circ(1_M\otimes f \otimes 1_A)\circ (\delta(m)\otimes a)\]
which commutes with the right $ A$-module structure. Further $ p_A:=\rho\otimes  1_A: M\otimes A\rightarrow M\otimes A$ is a homomorphism of right $A$-modules.

For each $ f\in \Hom_{\bfk}(H, A)$,
\begin{eqnarray*} P(f)\bullet p(m\otimes a)&=&(1_M\otimes m_A) \circ (1_M\otimes f \otimes 1_A)\circ( \delta\otimes 1_A)(m\otimes a)\\
&=& (1_M\otimes m_A) \circ [1_M\otimes f )\circ[((\rho\otimes 1 +1\otimes \sigma+\lambda 1_M\otimes 1_H))\circ \delta (m))\otimes a]\\
&=& p\big(P(f)\bullet  (m\otimes a)+f\bullet  p(m\otimes a)+\lambda f\bullet  (m\otimes a)\big).
\end{eqnarray*}
Hence $(M\otimes A, p)$ is a Rota-Baxter module for $( \Hom_{\bfk}(H, A), P)$.

This is parallel to the representation theory of group schemes, suggesting further exploration on the representation theory of Rota-Baxter algebra  schemes.

\subsection{Algebraic Birkhoff factorization for Rota-Baxter representations}

The algebraic Birkhoff factorization lies at the heart of the algebraic approach of Connes and Kreimer to renormalization of perturbative quantum field theory and its many applications in mathematics and physics~\mcite{CK1,CM,EGM,GZ}.

We first recall the general setup of Algebraic Birkhoff Decomposition~\mcite{CK1,EGK3,EGM}.  For any Rota-Baxter algebra $(A, Q)$, the $\bfk$-submodule $A_-:= \bfk+Q(A)$ is a subalgebra of $A$ and $Q(A_-)\subseteq A_-$. Thus $ (A_-, Q)$ is a Rota Baxter subalgebra of $(A, Q)$. Similarly, $A_+:=\bfk+\tilde{Q}(A)$ is a Rota-Baxter subalgebra of $(A, \tilde{Q})$. Then for any coalgebra $H$,
$H(A_-)$ is a Rota-Baxter subalgebra of $H(A)$. If $H$ is a Hopf algebra,
the subset $\Hom_{\bfk-\on{Alg}}(H, A)$ is  closed under the convolution product $\ast$ and $\Hom_{\bfk-\on{Alg}}(H, A_-) \subseteq \Hom_{\bfk-\on{Alg}}(H, A)$ is a sub-semigroup. Similarly, $\Hom_{\bfk-\on{Alg}}(H, A_+) \subseteq \Hom_{\bfk-\on{Alg}}(H, A)$  is a sub-semigroup.

Recall that a connected graded Hopf algebra~\mcite{CK1,Gub} is a Hopf algebra $H$ with grading $H=\oplus_{n\geq 0} H_n$ that is compatible with the multiplication and comultiplication of $H$ and such that $H_0=\bfk$.
Then we have the following {\bf algebraic Birkhoff factorization}~\mcite{CK1}.
\begin{theorem}
Let $H=\oplus_{n\geq 0} H_n$ be a connected graded Hopf algebra and $(A, Q)$ be a commutative Rota-Baxter algebra of weight $-1$ and $Q^2=Q$.
Then there is a map  $ \Hom_{\bfk\on{-Alg}}(H, A)\rightarrow  \Hom_{\bfk\on{-Alg}}(H, A_-)$, denoted by $ \varphi \mapsto \varphi_-$ such that  $\varphi_+:=\varphi_- \ast \varphi$ is in $\Hom_{\bfk\on{-Alg}}(H, A_+)$.
\mlabel{thm:abf}
\end{theorem}

The following consequence shows that the algebraic Birkhoff factorization is functorial in Rota-Baxter algebras.

\begin{lemma} If $ f: (A, Q_A)\rightarrow (A', Q_{A'})$ is a homomorphism of commutative Rota-Baxter algebras with $Q_A^2=Q_A$ and $Q_{A'}^2=Q_{A'}$, then for any connected graded Hopf algebra $H$, the map $ f_*: \Hom_\bfk(H, A)\rightarrow \Hom_\bfk(H,A')$ defined by $ f_*(\varphi)=f\circ \varphi$ has the following properties.
\begin{enumerate}
\item $ f_{*}(\Hom_{\bfk\on{-Alg}}(H,A))\subseteq \Hom_{\bfk\on{-Alg}}(H,A')$;
\item For any $ \varphi \in \Hom_{\bfk\on{-Alg}}(H,A) $, $f_*(\varphi_-)=(f_*(\varphi))_-$ and $f_*(\varphi_+)=(f_*(\varphi))_+$.
\end{enumerate}
\end{lemma}

We now consider an algebraic Birkhoff factorization for Rota-Baxter modules.
For a Rota-Baxter $(A,Q)$-module $(V,p_V)$, the $\bfk$-module $\Hom_{\bfk}(H, V)$  is an $A$-$H$-bimodule with
$$(a\psi h)(x):=a(\psi(hx)) \quad \text{for all } a\in A,  h\in H, x\in H, \psi\in \Hom_{\bfk}(H, V).$$
Then $\Hom_{\bfk}(H, A)$ acts on $ \Hom_{\bfk}(H, V)$ by
\begin{equation}\label{eq:RBM-star} (\varphi\ast \psi)(h)=\sum_{(h)}\varphi(h_{(1)})\psi(h_{(2)})
\end{equation}
for $ h\in H$ and $ \Delta(h)=\sum_{h}h_{(1)}\otimes h_{(2)}$.
With the linear operator
$$p': \Hom_{\bfk}(H, V)\longrightarrow \Hom_{\bfk}(H, V), \quad p'(\psi):=p_V\circ \psi \quad \text{for all } \phi\in \Hom_{\bfk}(H, V),
 $$
the pair $( \Hom_{\bfk}(H, V), p')$ is a Rota-Baxter module for the Rota-Baxter algebra $\Hom_{\bfk}(H,A)$.

 Fix a $\bfk$-algebra homomorphism $\varphi: H\rightarrow A$.  An element $\psi\in \Hom_{\bfk}(H,V)$ is called {\bf $ \varphi$-linear} if $ \psi $ is an $H$-module homomorphism, i.e., $ \psi(hx)=\varphi(h)\psi(x)$ for all $ h\in H$ and $x\in N$. We denote by $ \Hom_{\varphi}(H,V)$ the set of all $ \varphi$-linear elements.  Each $ \psi \in \Hom_{\varphi}(H,V)$ is uniquely determined by $\psi(1)\in V$.

 We note that for any fixed element $v\in V$, $ A_{-}v\subseteq V$ is an $ A_-$-submodule.  Similarly, $ A_+v$ is an $ A_+$-submodule of $V$.  Thus $\Hom_{\bfk}(H,  A_{-}v)$ is a module for the  algebra $ \Hom_{\bfk}(H,  A_{-})$ under the action $\ast$ defined by Eq.~\eqref{eq:RBM-star}. In particular $\Hom_{\bfk}(H,  A_{-}v)\subseteq \Hom_{\bfk}(H,  V)$.

\begin{theorem} Let $H$ be a connected Hopf algebra and $A$ a commutative Rota-Baxter algebra with idempotent operator $Q$.  Suppose  $ (V, p_V)$ is a $(A, Q)$-module and  $\varphi \in \Hom_{\bfk\on{-Alg}}(H, A)$.  For each $ \psi\in \Hom_{\varphi}(H, V)$, define $ \psi_+(h)=\varphi_+(h)\psi(1) $ and $\psi_-(h)=\varphi_-(h)\psi(1) $. Then  we have $\psi_+= \varphi_-\ast \psi \in \Hom_{\varphi_+}(H, A_+ \psi(1))$  and $ \psi_- \in \Hom_{\varphi_-}(H, A_- \psi(1))$.
\end{theorem}

\begin{proof} Since $ \varphi_-$ is an algebra homomorphism, there is $\varphi_-(hh')=\varphi_-(h)\varphi_-(h')$. We clearly have $ \psi_ -\in  \Hom_{\varphi_-}(H, A_{-} \psi(1))$. Similarly, $\psi_+\in  \Hom_{\varphi_+}(H, A_{+} \psi(1))$. We only need to verify $\psi_+=\varphi_-\ast \psi$ which  follows immediately from the algebraic Birkhoff factorization of $\varphi$:
$$
\hspace{3cm} \psi_+(h)=\varphi_+(h)\psi(1) =(\varphi_-\ast \varphi)(h)\psi(1)= (\varphi_-\ast \psi)(h). \hspace{3.5cm} \qedhere
$$
\end{proof}

\section{Rota-Baxter algebras for tensor categories}
\mlabel{sec:out}

We now generalize the concept of Rota-Baxter operators to be defined in a tensor category. Earlier we have been working in the symmetric tensor category $\calT$ of $\bfk$-modules.  We now assume that $\calT$ is a strict symmetric additive tensor category (strict monoidal category with a braiding $
b_{X,Y}: X\otimes Y\rightarrow Y\otimes X$ such that $ b_{X,Y}\circ b_{Y,X}=\on{Id}_{Y\otimes X}$).  The additive structure gives an abelian group structure on the set of morphisms between two objects.  We refer the readers to \cite{Ma} for background on monoidal tensor categories.

An algebra object (or monoid object) in $\calT$ is an object
$A$ together with a morphism $m: A\otimes A\rightarrow A$.
$m$ is associative if
$ m\circ (m\otimes 1)=m\circ(1\otimes m)$.
A Rota-Baxter object of weight $\lambda$ is a triple $ (A, m, P)$ such that
$ (A,m)$ is an algebra and
$ P\in \End_{\calT}(A) $ such that the diagram
\begin{equation}\xymatrix{ A\otimes A\ar[d]_{P\otimes 1+1\otimes P+\lambda\circ (1\otimes 1)} \ar[r]^{P\otimes P}&A\otimes A\ar[d]^{m}\\
A\otimes A \ar[r]_{p\circ m} &A .
}\end{equation}
Here $ \lambda$ is an element of endomorphism ring of the identity functor $\on{Id}: \calT\rightarrow \calT$.

If $ m$ is associative, then $ (A, m, p)$ is a Rota-Baxter associative algebra object. Similarly, $(A,m,p)$ is a Rota-Baxter Lie algebra object if $ m$ is a Lie algebra bracket, i.e.,
it satisfies the condition  $m\circ b_{A,  A} =-m$ and  the Jacobian identity
\[ m\circ (m\otimes 1)\circ (1+b_{123}+b_{123}^2)=0.\]
Here $ b_{123}=(1\otimes b_{A,A})\circ (b_{A, A}\otimes 1)$.

Given any associative algebra object $(A, m)$, there is a natural Lie algebra object $ (A_{Lie}, [,])$ with $ A_{Lie}=A$ as an object in $\calT$ and $[\cdot,\cdot]=m\circ (1-b_{A, A})$. Thus we have a functor  from the category $\on{Ass(\calT)} $ of associative algebra objects in $ \calT$ to the category $\on{Lie}(\calT)$ of Lie algebra objects in $ \calT$ (with morphisms be morphisms in $ \calT $ that commutes with the structure morphisms). However the existence of the left adjoint functor from $ \on{Lie}(\calT)$ to $ \on{Ass}(\calT)$ would depend on the category $ \calT$.

If $(A,m, P)$ is a Rota-Baxter associative algebra object of weight $\lambda$, then $ (A_{Lie}, [,], P)$ is a Rota-Baxter Lie algebra object.

Recall that for an associative algebra object $(A,m)$, an $A$-module is an object $M$ in $\calT$ together with a morphism
$\sigma: A\otimes M \rightarrow M$ satisfying
\begin{equation}
\sigma\circ (1\otimes \sigma)=\sigma \circ(m\otimes 1).
\end{equation}
If $(L, [,])$ is a Lie algebra object in $\calT$, then an
$L$-module is an object $M$ in $ \calT$ together with a
morphism $\sigma: L\otimes M\rightarrow M$ such that
\[ \sigma\circ ([,]\otimes 1)=\sigma\circ(1\otimes \sigma)\circ((1-b_{L,L})\otimes 1).\]

For a Rota-Baxter associative algebra object $(A, P)$ in $\calT$,  an $(A, P)$-module is an $A$-module  $M$ together with a morphism $p: M\rightarrow M$ in $\calT$ such that
\begin{equation}\xymatrix{ A\otimes M\ar[d]_{P\otimes 1+1\otimes p+\lambda\circ \sigma} \ar[r]^{P\otimes P}&A\otimes M\ar[d]^{\sigma}\\
A\otimes M\ar[r]_{p\circ\sigma} &M .
}\end{equation}
The compatibility condition in Eq.~\eqref{eq:associativity} is now
\begin{equation}
\sigma\circ(1\otimes \sigma)\circ (P\otimes P \otimes p)=
\sigma\circ (P\otimes p)\circ (m\otimes 1)\circ ((P\otimes 1+1\otimes P+\lambda(1\otimes 1))\otimes 1)
\end{equation}
which can be verified directly.

Let us give some examples of symmetric tensor categories, in addition to the category $\bfk\Mod$ of all $ \bfk$-modules, where Rota-Baxter algebras and Rota-Baxter modules might be fruitfully studied.

(a) The category of all $\ZZ$-graded  $\bfk$-modules with graded tensor product. This category has two different braidings, one with the standard switching of tensor factors, and another with change of sign $b_{X, Y}(x\otimes y)=(-1)^{ij}(y\otimes x) $ if $x\in X_i$ and $y\in Y_j$ are homogeneous elements.  The corresponding Lie algebras and associative algebras in two different braidings are different. With the first standard braiding, they are graded Lie algebras and graded associative algebras. But in the second signed braiding, the Lie algebra objects are in the super Lie algebra setting.  The Rota-Baxter algebra objects in these tensor categories as well as their representation theory  are very interesting topics. The first standard braiding is closely related to sheaves of the projective varieties. For the second case, taking the $ \ZZ/2\ZZ$-grading, one gets a super Rota-Baxter theory.

(b) The category of all differential graded $\bfk$-modules, whose objects are cochain complexes of $\bfk$-modules. The associative algebra objects are differential graded algebras and the Lie algebra objects are dg Lie algebras. Thus the above discussions also establishes the Rota-Baxter dg associative algebras and  dg Lie algebras.

(c) Let $X$ be a smooth algebraic variety over a field $\bfk$. The category $ \on{Coh}(X)$ of coherent sheaves on $X$ and its bounded derived category $D^b(\on{Coh}(X))$ are  symmetric  tensor categories. It  is interesting to consider  Rota-Baxter algebra objects in this category.  The Rota-Baxter algebras structures on the algebra $\Omega^{\bullet}(X)$ of differential forms should be very interesting and has been discussed in connection with singular hypersurfaces and renormalization on Kausz compactifications~\cite{MN}.

There are many other interesting symmetric tensor categories that have appeared in geometry and topology, as well as mathematical physics. It is interesting to interpret the Rota-Baxter algebra objects in those contexts as well.

Finally we remark that many properties do not require the symmetric property of the tensor category $\calT$ if one is limited to associative algebras only (not Lie algebras). Then one can consider quantum Rota-Baxter algebras by considering the tensor categories corresponding to solutions of Yang-Baxter equations. See~\cite{GL,Ji} for braided Rota-Baxter algebras whose module theory is to be developed.

Categorification of Rota-Baxter algebras has also been considered in \cite{CD} in terms of distributive monoidal category with a duality functor and an endo-functor so that the Grothendieck ring gives a Rota-Baxter algebra. Examples provided there have interesting geometric and topological applications and should be pursued further. This categorification might be related to the categorification in the context of 2-categories which is still to be developed.

\end{document}